\documentclass[]{article}

\usepackage[x11names]{xcolor}
\usepackage{graphicx} 
\usepackage{amsmath}
\usepackage{amsthm}
\usepackage{amssymb}
\usepackage{amsfonts}
\usepackage{tikz}
\usepackage{subcaption}
\usepackage{natbib}
\usepackage{mathtools}
\usepackage{dsfont}
\usetikzlibrary{plotmarks}
\usepackage{lmodern}

\definecolor{mycolor}{named}{Orchid4}

\usetikzlibrary{patterns}
\usetikzlibrary{matrix}
\usetikzlibrary{calc}

\usepackage{comment}

\usepackage[ruled,vlined]{algorithm2e}

\newtheorem{definition}{Definition}
\newtheorem{example}[definition]{Example}
\newtheorem{lemma}[definition]{Lemma}
\newtheorem{cor}[definition]{Corollary}
\newtheorem{theorem}[definition]{Theorem}
\newtheorem{problem}[definition]{Problem}

\DeclareMathOperator{\hcone}{hcone}
\DeclareMathOperator{\vcone}{vcone}
\DeclareMathOperator{\R}{\mathds{R}}
\DeclareMathOperator{\interior}{int}
\newcommand{\epreceqq}{\preceqq_{\!\!\varepsilon}}
 
\newcommand{\preceqq}{\;\raisebox{-0.12em}{\scalebox{1.15}{\ensuremath{\mathrel{\substack{\prec\\[-.15em]=}}}}}\;}

\DeclareMathOperator{\rank}{rank}

\usepackage[hidelinks]{hyperref}
\usepackage{url}
\usepackage{authblk}

\title{Representing the Non-dominated Set of Multi-objective Optimization Problems with Polyhedral Ordering Cones}
\author{Chiara Weuste}
\author{Michael Stiglmayr}
\affil{School of Mathematics and Natural Sciences, \\University of Wuppertal, Wuppertal, Germany}
\date{}

\begin{document}

\maketitle
\begin{abstract}
In the literature, research on representing the non-dominated set has focused almost exclusively on multi-objective optimization problems with Pareto-dominance. However, general ordering cones are used to incorporate preferences of decision makers into the optimization process and to model other specific types of optimization problems, such as ordinal optimization problems. In this work, we consider several representation approaches based on representation relations in the context of multi-objective optimization problems with polyhedral ordering cones. Our main research question is whether existing theory and solution approaches developed for the Pareto case can be used to address the more general setting.

\end{abstract}

\section{Introduction}

Many real-life optimization problems involve multiple, often conflicting objectives that have to be considered simultaneously. For example, a decision maker is not only interested in minimizing costs but also in optimizing additional criteria such as sustainability, quality, or social impacts. Such problems can be formulated as multi-objective optimization problems (MOP), where each goal is formulated as an objective function. Usually, the objective functions are in conflict, i.\,e., there does not exist a solution minimizing all objectives. Accordingly, to compare multidimensional outcome vectors one needs a new concept of optimality, e.\,g., \emph{Pareto optimality}, namely all solutions that cannot be improved in one objective function without worsening at least one other objective. Dominance relations, which are compatible with addition and scalar multiplication, can be expressed using ordering cones. Then, a solution dominates the region in the objective space given by the ordering cone attached to its objective vector.

In contrast to the majority of the existing literature, which focuses on Pareto-domi\-nance, this paper considers dominance relations defined by general cones, more precisely, pointed polyhedral cones. This generalized framework is particularly relevant in applications where decision makers express preferences regarding the relative importance of objectives, described by preference cones \citep{hunt10relative,nasrabadi-2019,wu-2025}. Moreover, several optimization problems can be reformulated into MOP with non-Pareto ordering cones. A notable example is provided by ordinal optimization problems, where costs are not represented by additive numerical values but by ordered categories that express qualitative evaluations \citep{schaefer2019shortest,klamroth2022matroid,klamroth2023reformulation,figueira2025computational}.

A fundamental challenge of discrete MOP is that the number of efficient solutions can grow exponentially with the input size \citep[see e.g.,][]{ehrgott-2005}. In continuous settings, the cardinality of the efficient set may even be infinite. Consequently, presenting the complete solution set to a decision maker is often impractical or even impossible. Instead, it is desirable to compute a small subset of solutions that adequately represents the efficient set, enabling the decision maker to select a preferred solution from a manageable number of alternatives. This problem is commonly referred to as the representation problem.

In this paper, we investigate representation problems for MOP with pointed polyhedral ordering cones. Building upon existing approaches developed for Pareto-dominance MOP \citep{sayn2000representation,vaz15representation,humbertropers25representation}, we adapt and extend these concepts to the more general setting.

The remainder of this paper is organized as follows. In Chapter~\ref{sec:ConesAndRelations}, we introduce the fundamental concepts and notation required throughout the paper. Chapter~\ref{sec:Representations} discusses the notion of a good representation and presents several representation relations that provide a formal characterization of the representation problem. The focus here is on investigating how existing methods for the Pareto ordering cone can be extended to more general ordering cones. Subsequently, Chapter~\ref{sec:Outlook} concludes the paper with an outlook on promising future research fields.

\section{Cones and Relations} \label{sec:ConesAndRelations}
This Chapter provides an overview of the basic concepts and theory of multiobjective optimization problems, with a particular focus on ordering cones and dominance relations. A comprehensive introduction to multiobjective optimization is given, e.g., in \cite{sawaragi85theory, ehrgott-2005}.
\begin{definition}
    A nonempty set \(C\subseteq \R^p\) is called a \emph{cone} if \(\lambda\, C \subset C\) for all \(\lambda>0\). A cone \(C\) is called
    \begin{itemize}
        \item \emph{convex} if \(\lambda\,c+ (1-\lambda) \, \bar{c}\in C\) for all \(c,\bar{c}\in C,\) and for all \(0<\lambda<1\),
        \item \emph{polyhedral} if there exists a matrix \(A_C\in \R^{m\times p}\setminus\{0\}\) such that \[C=\hcone(A_C)\coloneqq\{y\in\R^p:A_C\,y\geqq0\},\]
        \item \emph{pointed} if for all \(c\in C,\,c\neq 0 \implies (-c)\notin C\).
    \end{itemize} 
\end{definition}

A polyhedral cone is defined by a matrix whose row vectors correspond to the normal vectors of the facets that define the cone. Thus, a polyhedral cone can be seen as the intersection of \(m\) halfspaces. Accordingly, every polyhedral cone is convex. Alternatively, polyhedral cones can be described by their extreme rays.

\begin{theorem}[Weyl-Minkoski Theorem, cf.~\citealt{ziegler1995lectures}]
    A cone \(C\subseteq\R^p\) is finitely generated by \(n\) vectors in \(\R^p\), i.\,e., 
    \begin{align*}
         C=\vcone(B_C)\coloneqq\{B_C\,\lambda \colon \lambda\in\R^n,\;\lambda\geqq 0\} \text{ for some } B_C\in\R^{p\times n}
    \end{align*}
    if and only if it is a finite intersection of \(m\) halfspaces in \(\R^p\), i.\,e.,
    \begin{align*}
        C= \hcone(A_C)\coloneqq\{y\in\R^p:A_C\,y\geqq 0\} \text{ for some } A_C\in\R^{m\times p}.
    \end{align*}
\end{theorem}

The columns of matrix \(B_C\) correspond to the extreme rays of the cone \(C\). Note that the matrices \(A_C\) and \(B_C\) for a given cone \(C\) are not uniquely determined. 

In slide abuse of the notation, whenever we write in the following the sum of two sets of vectors, we mean the Minkoski sum of the two sets, i.\,e., every vector in the first set is added to every vector in the second set. We use the analogous convention for subtraction.

\begin{definition}[cf.~\citealt{engau-2007}]
    Let \(Y\subset\R^p\) be a nonempty set and let \(C\subset\R^p\) be a convex, pointed cone.
    \begin{align}
        N(Y,C)&\coloneqq\{y\in Y:(y-C\setminus\{0\})\cap Y =\emptyset\}, \label{eq:CNonDom}\\
        N_w(Y,C)&\coloneqq\{y\in Y \colon (y- \interior(C))\cap Y =\emptyset\},
    \end{align}
    are called the \emph{\(C\)-non-dominated} set and the \emph{weakly \(C\)-non-dominated} set of \(Y\), respectively. Thereby, \(\interior(C)\) denotes the interior of \(C\). Furthermore, we say that \(y\) \emph{\(C\)-dominates} \(\bar{y}\) denoted as \(y \preceq_{C} \bar{y}\) if \(y\in \bar{y}-C\setminus\{0\}\), and that \(y\) \emph{weakly \(C\)-dominates} \(\bar{y}\) denoted as \(y \preceqq_{\!\!C} \, \bar{y}\) if \(y\in \bar{y}-\interior(C)\).
\end{definition}

Consequently, \(N_w(Y,C)\) contains all points in \(Y\), which are not weakly \(C\)-dominated by any other point in \(Y\) and \(N(Y,C)\) contains all points in \(Y\), which are not \(C\)-dominated by any other point in \(Y\). Note that the zero vector is excluded in~\eqref{eq:CNonDom} to obtain an irreflexive ordering relation induced by cone \(C\):

\begin{align*}
    \mathcal{R}_C\coloneqq\{(y^1,y^2)\in\R^p\times\R^p:y^2-y^1\in C\}
\end{align*}

In the following we will denote pointed convex cones as \emph{ordering cones}. 

\begin{problem}
    The multi-objective optimization problem with ordering cone \(C\) can be written as
    \begin{align*}
        \textstyle\min_C \quad  &(f_1(x),\dots,f_p(x))^\top\tag{$\text{MOP}_C$}\label{eq:MOP_C}\\
        \mathrm{s.t.}\quad  &x\in X,
    \end{align*}
    with \(X\neq \emptyset\), the set of \emph{feasible solutions} in the \emph{decision space}, \(p\geq 2\) objectives (\(f_j:X\to\R\), \(j\in\{1,\dots,p\}\)) that should be optimized simultaneously. Without loss of generality, each objective should be minimized. Let \(Y=f(X)\subset\R^p\) denote the set of all \emph{outcome vectors} in the objective space. Note that there doesn't exist a canonical order of the elements in \(\R^p\), for \(p\geq 2\). Accordingly, an ordering has to be specified to compare the feasible points in the objective space and identify the optimal ones. The goal is to compute the \(C\)-non-dominated set defined by ordering cone \(C\). The corresponding points in the decision space are called \emph{\(C\)-efficient}.
\end{problem}

\begin{figure}[htbp]
    \centering
    \begin{minipage}[t]{0.45\linewidth}
        \centering
        \resizebox{\linewidth}{!}{
            \begin{tikzpicture}
                \fill[mycolor!40] (0,0) -- (2,0) -- (2,2) -- (0,2) -- (0,0);
                \draw[->] (-2,0) -- (2,0);
                \draw[->] (0,-0.5) -- (0,2);
                \node[mycolor!70!black] at (1,1) {\(P\)};
            \end{tikzpicture}
        } \caption{Pareto cone} \label{fig:ParetoCone}
    \end{minipage}
    \hfill
    \begin{minipage}[t]{0.45\linewidth}
        \centering
        \resizebox{\linewidth}{!}{
            \begin{tikzpicture}
                \fill[mycolor!40] (0,0) -- (2,0) -- (2,2) -- (-2,2) -- (0,0);
                \draw[->] (-2,0) -- (2,0);
                \draw[->] (0,-0.5) -- (0,2);
                \node[mycolor!70!black] at (1,1) {\(C_{o}\)};
            \end{tikzpicture}
        }  \caption{Ordinal cone} \label{fig:OrdinalCone}
    \end{minipage}
\end{figure} 
The most commonly used ordering cone in multi-objective optimization is the \emph{Pareto cone} given by 
\begin{align*}
    P\coloneqq\R^p_{\geq}=
    \bigl\{y\in\R^p \colon y\geqq0 \bigr\},
\end{align*}
which represents the componentwise order, see Figure~\ref{fig:ParetoCone}. For the Pareto cone \(P\) holds:

\begin{equation*}
    P=\hcone
    \begin{tikzpicture}[baseline={($(current bounding box.center)- (0,.8ex)$)}]
        \matrix (m) [matrix of math nodes,nodes in empty cells,right delimiter={)},left delimiter={(},inner sep=-2pt,nodes={inner sep=1.2ex}]{
        1 & & 0\\
        & &\\
        0&  & 1 \\
        } ;
        \draw[dotted,thick] (m-1-1)-- (m-3-3);
    \end{tikzpicture}
    = \vcone
    \begin{tikzpicture}[baseline={($(current bounding box.center)- (0,.8ex)$)}]
        \matrix (m) [matrix of math nodes,nodes in empty cells,right delimiter={)},left delimiter={(},inner sep=-2pt,nodes={inner sep=1.2ex}]{
        1 & & 0\\
        & &\\
        0&  & 1 \\
        } ;
        \draw[dotted,thick] (m-1-1)-- (m-3-3);
    \end{tikzpicture}.
\end{equation*}

A feasible solution \(x\) is called \emph{Pareto-optimal} or \emph{efficient} if the image point \(y=f(x)\in N(Y,P)\) is \(P\)-non-dominated. We will write in the following \(y\preceqq \bar{y}\), for the Pareto order \(y \preceqq_{\!\! P} \, \bar{y}\). Additionally a multi-objective optimization problem without specified ordering cone refers to a Pareto optimization problem denoted as MOP. The majority of the literature on multi-objective optimization addresses Pareto optimization problems. Consequently, this specific case has been studied much more thoroughly. The following result establishes a transformation between the non-dominated sets of \ref{eq:MOP_C}---a multi-objective optimization problem with polyhedral ordering cone \(C\)---and an associated Pareto optimization problem (MOP).

\begin{theorem}[see, e.g.,~\citealt{sawaragi85theory,engau-2007}] \label{thm:trans}
    Let \(Y\subset\R^p\) be a nonempty set and let \(C=\hcone(A_C)\) be a cone induced by matrix \(A_C\in\R^{m\times p}\). Then it holds
    \begin{align*}
        A_C\cdot N(Y,\hcone(A_C))\subseteq N(A_C\cdot Y,P).
    \end{align*}
    If the matrix \(A_C\) has full column rank, \(\mathrm{rank}(A_C)=p\), then both sets are equal, i.\,e., \(A_C\cdot N(Y,\hcone(A_C))=N(A_C\cdot Y,P)\).
\end{theorem}

Theorem~\ref{thm:trans} implies that any multi-objective optimization problem \ref{eq:MOP_C} with polyhedral ordering cone \(C\) and \(\mathrm{rank}(A_C)=p\) can be reformulated into a multi-objective optimization problem with respect to the Pareto cone \(P\subset\R^k\) with \(k\geq p\). The reformulated problem is given by:
\begin{align*}
    \mathrm{min}\quad  &A_C\cdot (f_1(x),\dots,f_p(x))^\top\\
    \mathrm{s.t.}\quad  &x\in X 
\end{align*}
A polyhedral cone \(C=\hcone(A_C)\) is pointed if and only if the matrix \(A_C\) has full column rank, i.\,e., \(\rank(A_C)=p\). 

\begin{example}[cf.~\citealt{klamroth2023reformulation}]
    We consider ordinal optimization problems (OOPs); ordinal objective functions are used in combinatorial optimization when the elements of a solution cannot be described by additive numerical values but are instead assigned to \(p\) ordered categories. These problems can be formulated as \ref{eq:MOP_C} with the ordinal cone \(C_o\) as the ordering cone, which is illustrated in Figure~\ref{fig:OrdinalCone} for the bi-objective case. The ordinal cone is a polyhedral ordering cone and is spanned by \(p\) linearly independent extreme rays. In particular, \(C_o=\hcone(A_o)\) with \(A_o\in\R^{p\times p}\) given by 
    \begin{equation*}
        A_o = \begin{tikzpicture}[baseline={($(current bounding box.center)- (0,.8ex)$)}]
        \matrix (m) [matrix of math nodes,nodes in empty cells,right delimiter={)},left delimiter={(},inner sep=-2pt,nodes={inner sep=1.2ex}]{
        1 & & 1\\
        & &\\
        0&  & 1 \\
        } ;
        \draw[dotted,thick] (m-1-1)-- (m-3-3);
        \draw[dotted,thick] (m-1-1)-- (m-1-3);
        \draw[dotted,thick] (m-1-3)-- (m-3-3);
    \end{tikzpicture}. 
    \end{equation*} 
    Since \(\rank(A_o)=p\), we can transform the ordinal optimization problem
    \begin{align*}
        \textstyle\min_{C_o} \quad  &f(x)= \big(f_1(x),\dots,f_p(x)\big)^\top \tag{$\text{OOP}$}\label{eq:OOP}\\
        \mathrm{s.t.}\quad  &x\in X 
    \end{align*}
    into the following Pareto optimization problem (MOP):
    \begin{align*}
        \textstyle\min \quad  &A_o\cdot c(x)= \Big(\sum_{i=1}^p c_i(x), \;\sum_{i=2}^p c_i(x),\;\ldots,\;c_p(x)\Big)^\top\\
        \mathrm{s.t.}\quad  &x\in X 
    \end{align*}
    A self-contained introduction to ordinal cones and ordinal optimization can be found in \cite{schaefer2019shortest}, \cite{klamroth2023reformulation}, \cite{klamroth2022matroid} and \cite{figueira2025computational}.
\end{example}
With the reformulation of a \ref{eq:MOP_C} in terms of a MOP the non-dominated set can be determined by applying the existing solution techniques for MOP on the reformulated problem and afterwards transforming the non-dominated set back into the original space by using \(B_C\coloneqq A^{-1}_C\), the inverse of the transformation matrix. This gives us a method for solving \ref{eq:MOP_C} with polyhedral ordering cones.

In the following, we will write short \(Y_N\) for all \(C\)-non-dominated points of a \ref{eq:MOP_C} and \(Y_{wN}\) for all weakly \(C\)-non-dominated points (\(Y_N\subseteq Y_{wN}\subseteq Y\)).

\section{Representation Problems} \label{sec:Representations}

In the previous Chapter, we presented a method for determining the non-dominated set of a multi-objective optimization problem with a polyhedral ordering cone (\ref{eq:MOP_C}). However, providing the decision maker with the entire set of non-dominated points is often impractical, as its cardinality can grow exponentially with the input size for discrete optimization problems and can even be infinite for continuous optimization problems. A suitably chosen subset that adequately represents the non-dominated set is often preferred, in particular when human decision makers are involved. In the following, we investigate the subset selection problem to determine a suitable representation, i.\,e., we assume that the non-dominated set \(Y_N\) is known \emph{a priori}. 
Our goal is to compute a \emph{representation} \(R\subseteq Y_N\) that best approximates the non-dominated set. This raises the question of what constitutes a good representation. Which properties should a subset of the non-dominated set satisfy in order to adequately reflect the structure of the non-dominated set? Although there doesn't exist a universally accepted definition in the operations research literature, most quality indicators are based on the following principles introduced in \cite{faulkenberg2009representations} and \cite{sayn2000representation}:
\begin{description}
    \item[Cardinality] The number of selected points should be minimized or alternatively bounded by a manageable number \(k\in\mathbb{Z}_{>}\), to ensure tractability.
    \item[Coverage] Every point in the non-dominated set should be adequately represented by at least one point in the representation.
    \item[Spacing] The points in the representation should be well distributed to avoid redundancy and ensure diversity. 
\end{description}

These three criteria are generally conflicting and should be balanced appropriately \citep[cf.][]{vaz15representation}. If there are not any bounds on the three indicators specified \emph{a priori}, each of them can be treated as an objective function, making the representation problem itself a three-objective optimization problem. Coverage and Spacing were first defined based on distance measures. Later, \cite{humbertropers25representation} generalized this approach by using a valued binary \emph{representation relation} \(\mathcal{P}\), which specifies for each pair \(y,\bar{y}\in Y\) the \emph{degree of acceptability} \(\mathcal{P}_{y\bar{y}}\) of the assertion ``\(y\) represents \(\bar{y}\)''. Here, a value of \(0\) denotes total acceptability, and as the degree increases, this representation becomes less accurate. Note that representation relations should be reflexive, i.e., \(\mathcal{P}_{yy}=0, \, \forall y\in Y\). This ensures that each point can represent itself. The general form for \emph{representation problems} is given by:

\begin{align*}
    \min \quad&|R| \\
    \min \quad&\mathrm{Coverage}(R,Y_N) \coloneqq \max_{y\in Y_N}\min_{r\in R} \mathcal{P}_{r,y} \label{eq:representationProblem} \tag{RP}\\
    \max \quad&\mathrm{Spacing}(R)\coloneqq\min_{\substack{r, \bar{r} \in R \\ r \neq \bar{r}}}\mathcal{P}_{r,\bar{r}}\\
    \mathrm{s. \,t.} \quad &R\subseteq Y_N,
\end{align*}

The coverage value for a given non-dominated set~\(Y_N\) and a given representation~\(R\) is determined by assigning to each point~\(y\) of the non-dominated set the point~\(r\) in the representation which represents it best, i.\,e., for which the degree of acceptability for \(r\) representing \(y\) is minimal. The coverage value is then determined by the worst represented point \(y\in Y_N\), the point for which the best representing point in \(R\) has the highest degree of acceptability. The spacing value for a given representation is defined by the minimal degree of acceptability for one point in the representation representing another point in the representation. 

Instead of considering the three-objective representation problem, it is often more reasonable to transform one or two objectives into constraints. This is also described in \cite{humbertropers25representation}, where an upper bound on the cardinality is fixed, and then coverage and spacing are optimized lexicographically, i.\,e., the spacing is optimized over all representations, which minimize the coverage value. However, we will continue to consider the representation problem from a three-criterion perspective as the described constrained versions coincide with the corresponding \(\varepsilon\)-constraint scalarizations.

The definitions for coverage and spacing depend on the considered representation relation and the corresponding degree of acceptability. In the following, we investigate different representation relations for MOP and examine which of them can be extended to the case of \ref{eq:MOP_C} with polyhedral ordering cones. A straightforward approach is to transform the \ref{eq:MOP_C} into an MOP instance using the transformation matrix \(A_C\), that induces the ordering cone \(C\), i.\,e., \(C=\hcone(A_C)\), and solve the transformed problem in order to obtain the non-dominated set \((A_C\cdot Y)_N\) as discussed in the previous chapter. Subsequently, the representation problem~\ref{eq:representationProblem} is solved based on \((A_C\cdot Y)_N\), yielding a representation \(A_C\cdot R\) for the corresponding Pareto optimization problem.

The \ref{eq:representationProblem} transformed in the Pareto structure is given by:

\begin{align*}
    \min \quad&|A_C\cdot R| \\
    \min \quad&\max_{y\in (A_C\cdot Y)_N}\min_{r\in A_C\cdot R} \mathcal{P}_{r,y} \\
    \max \quad&\min_{\substack{r, \bar{r} \in A_C\cdot R \\ r \neq \bar{r}}}\mathcal{P}_{r,\bar{r}} \tag{RP$_P$}\label{eq:RPp}\\
    \mathrm{s. \,t.} \quad &A_C\cdot R\subseteq (A_C\cdot Y)_N
\end{align*} 

The obtained optimal representation for \ref{eq:RPp} can be mapped back on the original optimization problem with polyhedral ordering cone using the matrix \(B_C\): \(B_C\cdot A_C\cdot R=R\), which constitutes a subset of the non-dominated set for the original problem. However, it remains to be investigated whether this is an optimal representation for the optimization problem with a general polyhedral ordering cone.

We will analyze for which representation relations the problems \ref{eq:RPp} of the corresponding Pareto optimization problem and the representation problem \ref{eq:representationProblem} of the original problem are equivalent.  If both problems are equivalent, existing algorithms for \ref{eq:RPp} can be applied to the transformed problem, and the resulting representation can be mapped back. Otherwise, alternative approaches must be developed to address the problem directly on the non-dominated set~\(Y_N\) of \ref{eq:MOP_C}. 

In the following, we consider various concepts for representation relations, each of which may yield different optimal representations for the same problem. Two of them are dominance-based approaches, relying on \(\varepsilon\)-dominance and \(\gamma\)-dominance, while the others are distance-based representation relations. Each of these concepts has its own advantages and limitations in terms of interpretability, applicability, and computational tractability.

\subsection{Distance-based Representation Relations} \label{sec:Distances}

The first representation relation is based on a metric in the objective space. For a given metric \(d\colon \R^p \times \R^p \rightarrow \R\), the \emph{distance-based degree of acceptability} of a point \(y\in\R^p\) representing another point \(\bar{y}\in\R^p\) is \( \delta_{y,\bar{y}}\coloneqq d(y,\bar{y})\).
Accordingly, the distance-coverage value of a give representation \(R\) for a non-dominated set \(Y_N\) is the maximal distance from a point of the non-dominated set to its closest point in the representation, i.\,e.,
\begin{align*}
    \max_{y\in Y_N} \min_{r\in R} \delta_{y,r}.
\end{align*}
The spacing value of a representation \(R\subset Y_N\) is the minimal distance between any two points in \(R\), i.\,e., 
\begin{align*}
    \min_{\substack{r, \bar{r} \in R \\ r \neq \bar{r}}} \delta_{r,\bar{r}}.
\end{align*}
Incorporating the distance-based degree of acceptability into the \ref{eq:representationProblem} yields the distance-based representation problem. 
In the following we are investigating whether the representation problem can be solved equivalently for the transformed Pareto problem. In general, by linearly transforming the points in \(Y_N\) with the matrix \(A_C\), the relative distances between the points change (as \(A_C\) is not necessarily a length-preserving transformation). As a result, points that are close together in the original problem may be farther apart after the transformation to the Pareto problem. Consequently, when solving a \ref{eq:MOP_C} via reformulation in the Pareto structure, the non-dominated set must be transformed back, as the representation problem \eqref{eq:RPp} on the non-dominated set of the associated Pareto problem \(A_C\, Y_N\) and the representation problem \eqref{eq:representationProblem} on the original non-dominated set \(Y_N\) are not equivalent.

Up to now, we chose the metric in the representation relation independent of the respective ordering cone. This raises the question: How is the distance measure comply with different ordering cones, or in other words, is it appropriate to use the same distance measure for all ordering cones? Alternatively, we can define a cone-dependent distance measure:

\begin{definition}
    Let \(d\colon \R^p\times\R^p \rightarrow \R\) be a metric, and \(C\subseteq\R^p\) be a polyhedral cone, with \(C=\hcone(A_C)\) and \(A_C\in\R^{m\times p}\) being a row-normalized matrix. Then, we define the \emph{cone-induced distance} $d^C(y,\bar{y})$ between two points \(y,\bar{y}\in \R^p\) as:
    \begin{align*}
        d^C(y,\bar{y})\coloneqq d(A_C\cdot y,A_C\cdot \bar{y}).
    \end{align*}
\end{definition} 

The cone-induced distance measures the distance between the transformed points (in the associated Pareto problem). We use row-normalized matrices to reduce ambiguity as infinitely many matrices induce the same cone. By row-normalizing the matrix \(A_C\), we describe the cone \(\hcone(A_C)\) by its normalized normal vectors. Note that the matrix \(A_C\) is not uniquely determined, as any permutation of its rows generates the same \(\hcone(A_C)\). However, the cone-induced distance is invariant with respect to it. 

For example, the ordinal cone can be written as follows: 

\begin{equation*}
C_o = \hcone
\vcenter{\hbox{\begin{tikzpicture}[baseline={($(current bounding box.center)- (0,.8ex)$)}]
    \matrix (m) [matrix of math nodes,nodes in empty cells,right delimiter={)},left delimiter={(},inner sep=-2pt,nodes={inner sep=1.2ex}]{
    1 & & 1\\
    & &\\
    0&  & 1 \\
    } ;
    \draw[loosely dotted,thick] (m-1-1)-- (m-3-3);
    \draw[loosely dotted,thick] (m-1-1)-- (m-1-3);
    \draw[loosely dotted,thick] (m-1-3)-- (m-3-3);
\end{tikzpicture}}}
= \hcone
\vcenter{\hbox{\begin{tikzpicture}
    \matrix (n) [
        matrix of math nodes,
        nodes in empty cells,
        right delimiter={)},
        left delimiter={(},
        inner sep=-2pt,
        nodes={inner sep=1ex}
    ]{
        \frac{1}{\sqrt{p}} & & & \frac{1}{\sqrt{p}} & \frac{1}{\sqrt{p}}\\
         & \frac{1}{\sqrt{p-1}} & & \frac{1}{\sqrt{p-1}}& \frac{1}{\sqrt{p-1}}\\
        & & & & \\
        & & & \frac{1}{\sqrt{2}} & \frac{1}{\sqrt{2}}\\
        0 & & & & 1\\
    };

    \draw[loosely dotted,thick] (n-1-1)--(n-1-4);
    \draw[loosely dotted,thick] (n-2-2)--(n-4-4);
    \draw[loosely dotted,thick] (n-2-2)--(n-2-4);
    \draw[loosely dotted,thick] (n-2-4)--(n-4-4);
    \draw[loosely dotted,thick] (n-2-5)--(n-4-5);
\end{tikzpicture}}}
\end{equation*}
The second matrix is row-normalized and therefore used to define the cone-induced distance w.\,r.\,t.\ the ordinal cone. 

The corresponding Euclidean ordinal-distance is given by
\begin{align*}
    d_2^{C_o}(y,\bar{y})&=d_2\left(\begin{pmatrix}
        \frac{1}{\sqrt{p}}\sum_{i=1}^p y_i\\
        \frac{1}{\sqrt{p-1}}\sum_{i=2}^p y_i\\
        \vdots \\
        \frac{1}{\sqrt{2}}\sum_{i=p-1}^p y_i\\
        y_p
    \end{pmatrix}, \begin{pmatrix}
        \frac{1}{\sqrt{p}}\sum_{i=1}^p \bar{y}_i\\
        \frac{1}{\sqrt{p-1}}\sum_{i=2}^p \bar{y}_i\\
        \vdots \\
        \frac{1}{\sqrt{2}}\sum_{i=p-1}^p \bar{y}_i\\
        \bar{y}_p
    \end{pmatrix}\right)\\[2ex]
&=\sqrt{\sum_{k=1}^p\frac{1}{p+1-k}\Bigl(\sum_{i=k}^p\left(\bar{y}_i-y_i\right)\Bigr)^2}
\end{align*}
and the Manhattan ordinal-distance is given by
\begin{align*}
   d^{C_o}_1(y,\bar{y})\coloneqq \sum_{k=1}^p\frac{1}{\sqrt{p+1-k}} \,\Bigl|\sum_{i=k}^p \bar{y}_i-y_i\Bigr|.
\end{align*} 

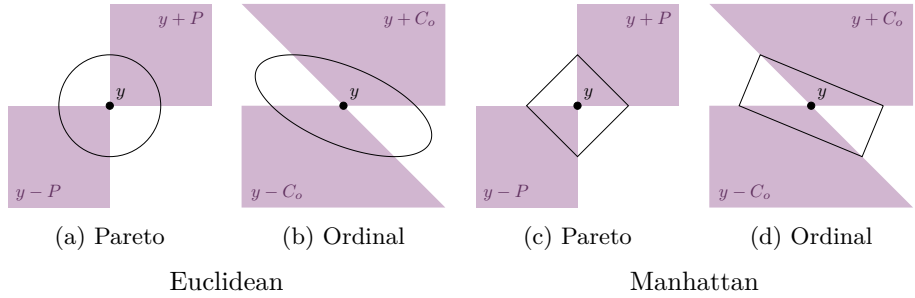
\begin{figure}[htbp]
    \centering
    \begin{minipage}[t]{0.49\linewidth}
        \centering
        \begin{subfigure}[t]{0.48\linewidth}
            \centering
            \resizebox{\linewidth}{!}{
                \begin{tikzpicture}
                    \fill[mycolor!40] (0,0) -- (2,0) -- (2,2) -- (0,2) -- (0,0);
                    \node[mycolor!70!black] at (1.4,1.7) {$y+P$};
                    \fill[mycolor!40] (0,0) -- (-2,0) -- (-2,-2) -- (0,-2) -- (0,0);
                    \node[mycolor!70!black] at (-1.4,-1.7) {$y-P$};
                    \filldraw (0,0) circle (2pt) node[above right] {$y$};
                    \draw[] (0,0) circle (1);
                \end{tikzpicture}
            }
            \caption{Pareto}
        \end{subfigure}
        \hfill
        \begin{subfigure}[t]{0.48\linewidth}
            \centering
            \resizebox{\linewidth}{!}{
                \begin{tikzpicture}
                    \fill[mycolor!40] (0,0) -- (2,0) -- (2,2) -- (-2,2) -- (0,0);
                    \node[mycolor!70!black] at (1.3,1.7) {$y+C_o$};
                    \fill[mycolor!40] (0,0) -- (-2,0) -- (-2,-2) -- (2,-2) -- (0,0);
                    \node[mycolor!70!black] at (-1.3,-1.7) {$y-C_o$};
                    \draw[domain=0:360, samples=200]
                        plot ({sqrt(2)*cos(\x) - sin(\x)}, {sin(\x)});
                    \filldraw (0,0) circle (2pt) node[above right] {$y$};
                \end{tikzpicture}
            }
            \caption{Ordinal}
        \end{subfigure}

        \medskip
        Euclidean
    \end{minipage}
    \hfill
    \begin{minipage}[t]{0.49\linewidth}
        \centering
        \begin{subfigure}[t]{0.48\linewidth}
            \centering
            \resizebox{\linewidth}{!}{
                \begin{tikzpicture}
                    \fill[mycolor!40] (0,0) -- (2,0) -- (2,2) -- (0,2) -- (0,0);
                    \node[mycolor!70!black] at (1.4,1.7) {$y+P$};
                    \fill[mycolor!40] (0,0) -- (-2,0) -- (-2,-2) -- (0,-2) -- (0,0);
                    \node[mycolor!70!black] at (-1.4,-1.7) {$y-P$};
                    \filldraw (0,0) circle (2pt) node[above right] {$y$};
                    \draw[] (1,0) -- (0,1) -- (-1,0) -- (0,-1) -- (1,0);
                \end{tikzpicture}
            }
            \caption{Pareto}
        \end{subfigure}
        \hfill
        \begin{subfigure}[t]{0.48\linewidth}
            \centering
            \resizebox{\linewidth}{!}{
                \begin{tikzpicture}
                    \fill[mycolor!40] (0,0) -- (2,0) -- (2,2) -- (-2,2) -- (0,0);
                    \node[mycolor!70!black] at (1.3,1.7) {$y+C_o$};
                    \fill[mycolor!40] (0,0) -- (-2,0) -- (-2,-2) -- (2,-2) -- (0,0);
                    \node[mycolor!70!black] at (-1.3,-1.7) {$y-C_o$};
                    \filldraw (0,0) circle (2pt) node[above right] {$y$};
                    \draw[] ({sqrt(2)},0) -- (-1,1) -- (-{sqrt(2)},0) -- (1,-1) -- ({sqrt(2)},0);
                \end{tikzpicture}
            }
            \caption{Ordinal}
        \end{subfigure}

        \medskip
        Manhattan
    \end{minipage}

    \caption{Unit circle for different cone-induced distances}
    \label{fig:unitCircles}
\end{figure} 
Figure~\ref{fig:unitCircles} illustrates the unit circle for different cone-induced distances in the bi-objective case. The point \(y\) dominates all points in \(y+C\) and all points in \(y-C\) dominate \(y\). Accordingly, if \(y\) is a non-dominated point other non-dominated points can only be found in the white space, i.\,e., in \(\R^p\setminus((y+C)\cup(y-C))\). For the Pareto case, the unit circle will not change when considering the cone-induced distance instead of the original distance. The cone-induced distance, however, incorporates the facets of the ordering cone and such the preference information into the distance metric. Thus, it is more reasonable in order to determine a suitable representation, than using a metric independent from the ordering cone.

With this definition, the distance-based degree of acceptability of a point \(y\) representing another point \(\bar{y}\), where \(y,\bar{y}\in\R^p\), with respect to an ordering cone \(C\), becomes
\begin{align*}
    \delta^C_{y,\bar{y}}\coloneqq d^C(y,\bar{y}).
\end{align*}

From the definition of the cone-induced distance it follows directly that \[\delta_{y,\bar{y}}^C=\delta_{A_C\cdot y,A_C\cdot \bar{y}}^P.\] Therefore, for a given polyhedral ordering cone \(C\) with \(A_C\) being a row-normalized matrix that induces the cone, a non-dominated set \(Y_N\) and a representation \(R\subseteq Y_N\) the following equalities hold:
\begin{align*}
    \max_{y\in Y_N} \min_{r\in R} \delta^C_{y,r}&=
    \max_{y\in (A_C\cdot Y)_N} \min_{r\in A_C\cdot R} \delta^P_{y,r}, \\
    \min_{r,\bar{r}\in R} \delta^C_{r,\bar{r}} &= \min_{r,\bar{r}\in A_C\cdot R} \delta^P_{r,\bar{r}}. 
\end{align*}
Accordingly, the representation problem with cone-induced distances can be equivalently formulated and solved for the corresponding  Pareto optimization problem. 

\subsection{Representation Relation Based on \(\varepsilon\)-Dominance} \label{sec:Epsilon}

The following two representation relations are dominance-based, meaning that a point \(y\) can represent another point \(\bar{y}\), if \(y\) ``dominates'' \(\bar{y}\). Therefore, a new dominance concept is required that relaxes cone-dominance. The \(\varepsilon\)-dominance representation achieves this by introducing a multiplicative tolerance parameter \(\varepsilon\geq0\). If we consider the Pareto cone as the ordering cone, we get the following definition:
\begin{definition}[cf.~\citealt{humbertropers25representation}] \label{def:eps-dom-pareto}
    Let \(\varepsilon\geq 0\) and \(y,\bar{y}\in \R_{\geq}^p\). The point \(y\) is said to \(\varepsilon\)-dominate \(\bar{y}\) (w.\,r.\,t.\ the Pareto cone \(P\subset\R^p\)), denoted by \(y\epreceqq \,\bar{y}\) if
    \begin{align*}
        y_i \leq (1+\varepsilon) \, \bar{y}_i \quad \forall i\in\{1,\dots,p\}.
    \end{align*}
\end{definition}
In other words, \(y\) \(\varepsilon\)-dominates \(\bar{y}\) if \(y\) is at least as good as \(\bar{y}\) within a factor of \((1+\varepsilon)\) in every objective. Equivalently, this can be interpreted geometrically by rewriting the inequality into \(\frac{y_i}{1+\varepsilon}\leq \bar{y}_i, \ \forall i\in\{1,\dots,p\}\). The point \(y\in Y_N\) is scaled towards the origin by the factor \(\frac{1}{1+\varepsilon}< 1\) (for \(\varepsilon>0\)), i.\,e., it is shifted along the line \(\{t\,y\colon t\in[0,1]\}\) connecting \(y\) with the origin. Therefore, \(y\) \(\varepsilon\)-dominates every point in \(\frac{y}{1+\varepsilon}+P\). 
Figure~\ref{fig:eps-dom-pareto} illustrates the \(\varepsilon\)-dominance for the Pareto cone in the bi-objective case. Note that for \(\varepsilon=0\), \(\varepsilon\)-dominance coincides with Pareto-dominance, except that a point \(\varepsilon\)-dominates itself. The binary relation associated with the \(\varepsilon\)-dominance for \(\varepsilon>0\) is reflexive and transitive, but neither asymmetric nor antisymmetric, as there are points \(y^1,y^2\in\R^p_\geq\) such that \(y^1 \epreceqq \, y^2\) and \(y^2 \epreceqq \, y^1\).
 
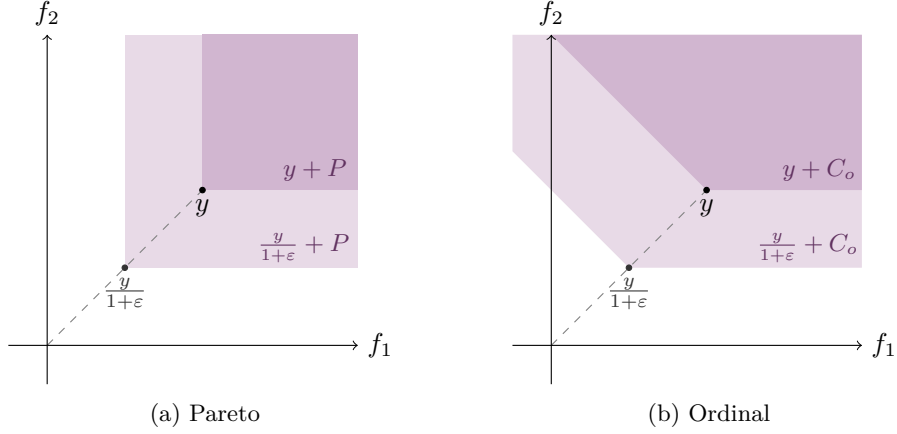
\begin{figure}[htbp]
    \centering
    \begin{subfigure}[t]{0.45\linewidth}
        \centering
        \resizebox{\linewidth}{!}{
            \begin{tikzpicture}
                \fill[mycolor!20] (1,1) -- (4,1) -- (4,4) -- (1,4) -- (1,1);
                \fill[mycolor!40] (2,2) -- (4,2) -- (4,4) -- (2,4) -- (2,2);
                \node[mycolor!70!black] at (3.45,2.26) {\small\(y+P\)};
                \node[mycolor!70!black] at (3.3,1.26) {\small\(\frac{y}{1+\varepsilon}+P\)};
                \draw[dashed, gray] (0,0) -- (2,2);
                \draw[->] (-0.5,0) -- (4,0) node[right] {\(f_1\)};
                \draw[->] (0,-0.5) -- (0,4) node[above] {\(f_2\)};
                \filldraw[] (2,2) circle (1pt) node[below ] {\(y\)};
                \filldraw[black!80] (1,1) circle (1pt) node[below] {\(\frac{y}{1+\varepsilon}\)};
            \end{tikzpicture}
        }
        \caption{Pareto}
        \label{fig:eps-dom-pareto}
    \end{subfigure}
    \hfill
    \begin{subfigure}[t]{0.45\linewidth}
        \centering
        \resizebox{\linewidth}{!}{
            \begin{tikzpicture} []
                \fill[mycolor!20] (1,1) -- (4,1) -- (4,4) -- (-0.5,4) -- (-0.5,2.5) -- (1,1);
                \fill[mycolor!40] (2,2) -- (4,2) -- (4,4) -- (0,4) -- (2,2);
                \node[mycolor!70!black] at (3.45,2.26) {\small\(y+C_o\)};
                \node[mycolor!70!black] at (3.3,1.26) {\small\(\frac{y}{1+\varepsilon}+C_o\)};
                \draw[dashed, gray] (0,0) -- (2,2);
                \draw[->] (-0.5,0) -- (4,0) node[right] {\(f_1\)};
                \draw[->] (0,-0.5) -- (0,4) node[above] {\(f_2\)};
                \filldraw[] (2,2) circle (1pt) node[below] {\(y\)};
                \filldraw[black!80] (1,1) circle (1pt) node[below] { \(\frac{y}{1+\varepsilon}\)};
            \end{tikzpicture}
        }
        \caption{Ordinal}
        \label{fig:eps-dom-ordinal}
    \end{subfigure}
    \caption{\(\varepsilon\)-dominated for the point \(y\) w.\,r.\,t.\ the Pareto and the ordinal cone.}
    \label{fig:eps-dom}
\end{figure} 
\begin{cor} \label{cor:eps-dom-pareto}
    Let \(\varepsilon\geq 0\) and \(y,\bar{y}\in\R_{\geq}^p\). If \(y\preceqq \bar{y}\), then \(y\epreceqq \,\bar{y}\). 
\end{cor} 

Corollary~\ref{cor:eps-dom-pareto} shows that Pareto-dominance is stronger than \(\varepsilon\)-dominance on \(\R_{\geq}^p\). Equivalently, for any \(y\in\R_{\geq}^p\), it holds
\begin{align*}
    y+ P\subseteq \frac{y}{1+\varepsilon}+ P \quad \forall \varepsilon\geq0.
\end{align*} If \(y\notin\R_{\geq}^p\), scaling \(y\) by \(\frac{1}{1+\varepsilon}\) yields a point that does not dominate \(y\). In this case, \(\varepsilon\)-dominance isn't constitutes a relaxation of Pareto-dominance and is therefore not suitable for defining a representation relation.

The \(\varepsilon\)-dominance can be expressed equivalently as
\begin{align*}
    y\epreceqq \, \bar{y} &\iff \frac{y_i}{\bar{y}_i}-1\leq \varepsilon  \quad \forall i\in\{1,\ldots,p\}.
\end{align*}
In the Pareto case, the \emph{degree of acceptability based on \(\varepsilon\)-dominance} for \(y\) representing \(\bar{y}\) is defined as the minimal value of \(\varepsilon\) such that \(y\) \(\varepsilon\)-dominates \(\bar{y}\), i.\,e.,
\begin{align*}
    \varepsilon_{y\bar{y}}\coloneqq \max\left\{\max_{i\in\{1,\dots,p\}}\left\{\frac{y_i}{\bar{y}_i}-1\right\},0\right\} = \max\left\{\max_{i\in\{1,\dots,p\}}\left\{\frac{y_i}{\bar{y}_i}\right\}-1,0\right\}.
\end{align*}

The concept of \(\varepsilon\)-dominance is now extended for general ordering cones.

\begin{definition} \label{def:eps-dom}
    Let \(\varepsilon\geq 0\), \(y,\bar{y}\in \R^p\) and \(C\subset\R^p\) be a ordering cone. The point \(y\) is said to \(\varepsilon\)-dominate \(\bar{y}\) w.\,r.\,t.\ the ordering cone \(C\), denoted by 
    \(y\preceqq_{C,\varepsilon}\, \bar{y}\)) if
    \begin{align*}
        y \preceqq_{\!\!C} \, (1+\varepsilon) \, \bar{y}.
    \end{align*}
\end{definition}

In Figure~\ref{fig:eps-dom-ordinal} \(\varepsilon\)-dominance is illustrated for a point in the bi-objective space with the ordinal cone as the ordering cone. Corollary~\ref{cor:eps-dom-pareto} does not extend to general cones. In particular, \(y\preceqq_{\!\!C} \, \bar{y}\) does not necessarily imply \(y\preceqq_{\!\!C,\varepsilon} \, \bar{y}\), as illustrated in Figure~\ref{fig:epsilonLemma}. This raises the question weather or under which assumptions cone-dominance implies the corresponding \(\varepsilon\)-dominance. Equivalently, under which assumptions does \(y+ C\subseteq \frac{y}{1+\varepsilon}+ C\) hold for every \(\varepsilon\geq0\)?

\begin{lemma}\label{lem:epsilonEquivalence}
    Let \(y\in\R^p\) and \(C\subset \R^p\) be a ordering cone. Then the following equivalence holds:
    \begin{align*}
        \forall\varepsilon\geq0: y+ C \subseteq \frac{y}{1+\varepsilon}+ C \iff y\in C
    \end{align*}
\end{lemma}

\begin{proof}\quad
    \(
    \begin{aligned}[t]
        &\forall\varepsilon\geq0:\, y+ C \subseteq \frac{y}{1+\varepsilon}+ C \\
        &\iff \forall\varepsilon\geq0: (1+\varepsilon)\,y+ C \subseteq y+ C \\
        &\iff  \forall\varepsilon\geq0:(1+\varepsilon)\,y\in y+ C\\
        &\iff \forall\varepsilon\geq0: y+\underbracket[0.5pt]{\varepsilon \,y}_{\in C}\in y+ C \\
        &\iff y\in C 
    \end{aligned}\)\\[-2.5ex]\qedhere
\end{proof}

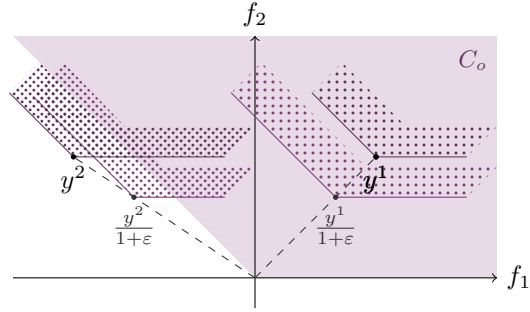
\begin{figure} [htbp] \centering
\begin{tikzpicture} [scale = 0.8]
    \fill[mycolor!20] (0,0) -- (4,0) -- (4,4) -- (-4,4) --(0,0);
    \draw[->] (-4,0) -- (4,0) node[right] {$f_1$};
    \draw[->] (0,-0.5) -- (0,4) node[above] {\(f_2\)};
    \node[mycolor!70!black] at (3.6,3.6) {\small\(C_o\)};

    \draw[mycolor!70!black] (2,2) -- (3.5,2);
    \draw[mycolor!70!black] (2,2) -- ({2-sqrt(1.125)}, {2+sqrt(1.125)});
    \filldraw[] (2,2) circle (1pt) node[below] {\(y^1\)};
    \fill[pattern=dots, pattern color=mycolor!70!black]({2-sqrt(1.125)}, {2+sqrt(1.125)}) -- (2,2) -- (3.5,2) -- (4,2.5) -- (2.5,2.5) -- ({2.5-sqrt(1.125)}, {2.5+sqrt(1.125)}) -- cycle;
    \draw[mycolor!90!black] (4/3,4/3) -- (3.5,4/3);
    \draw[mycolor!90!black] (4/3,4/3) -- ({4/3-(2+sqrt(1.125)-4/3)},{4/3+(2+sqrt(1.125)-4/3)});
    \fill[pattern=dots, pattern color=mycolor!90!black]({8/3-(2+sqrt(1.125))},{4/3+(2+sqrt(1.125)-4/3)}) -- (4/3,4/3) -- (3.5,4/3) -- (4,4/3+0.5) -- (4/3+0.5,4/3+0.5) -- ({8/3-(2+sqrt(1.125))+0.5},{2.5+sqrt(1.125)}) -- cycle;
\draw[dashed, black!80] (0,0) -- (2,2);
    \filldraw[black!80] (4/3,4/3) circle (1pt) node[below] {\(\frac{y^1}{1+\varepsilon}\)};

    \draw[mycolor!90!black] (-2,4/3) -- (-0.5,4/3);
    \draw[mycolor!90!black] (-2,4/3) -- (-5+4/3,3);
    \fill[pattern=crosshatch dots, pattern color=mycolor!90!black](-5+4/3,3) -- (-2,4/3) -- (-0.5,4/3) -- (0,4/3+0.5) -- (-1.5,4/3+0.5) -- ({-4.5+4/3},{2.5+sqrt(1.125)}) -- cycle;
    \fill[pattern=crosshatch dots, pattern color=mycolor!70!black](-0.5,2) -- (-3,2) -- ({-3-sqrt(1.125)},{2+sqrt(1.125)}) -- ({-2.5-sqrt(1.125)},{2.5+sqrt(1.125)}) -- (-2.5,2.5) -- (0,2.5) -- cycle;
\draw[dashed, black!80] (0,0) -- (-3,2);
    \draw[mycolor!70!black] (-3,2) -- (-0.5,2);
    \draw[mycolor!70!black] (-3,2) -- ({-3-sqrt(1.125)},{2+sqrt(1.125)});
    \filldraw[] (-3,2) circle (1pt) node[below] {\(y^2\)};
    \filldraw[black!80] (-2,4/3) circle (1pt) node[below] {\(\frac{y^2}{1+\varepsilon}\)};

    \filldraw[] (2,2) circle (1pt) node[below] {\(y^1\)};
\end{tikzpicture}
\caption{\(\varepsilon\)-dominance for points inside (\(y^1\)) and outside (\(y^2\)) the ordering cone \label{fig:epsilonLemma}}
\end{figure} 
Lemma~\ref{lem:epsilonEquivalence} shows that the \(\varepsilon\)-dominance is a relaxation of the cone-dominance if and only if the corresponding point \(y\) itself lies in the ordering cone~\(C\). This result is illustrated in Figure~\ref{fig:epsilonLemma}, where the ordinal cone is considered. The point \(y^1\) lies in the cone. Consequently, cone-dominance is stronger than \(\varepsilon\)-dominance for every \(\varepsilon>0\). In contrast, \(y^2\) does not lie in the ordering cone. Hence, cone-dominance is not necessarily implied by \(\varepsilon\)-dominance. Thus, using \(\varepsilon\)-dominance as basis of a representation relation is only meaningful if the non-dominated set of \ref{eq:MOP_C} is contained in the ordering cone, i.\,e., \(Y_N\subset C\). 

For a given ordering cone \(C\) the \emph{\(\varepsilon\)-based degree of acceptability} for \(y\) representing \(\bar{y}\) is defined by
\begin{align*}
    \varepsilon_{y\bar{y}}^C\coloneqq \max\left\{\min\{\varepsilon:y \preceqq_{\!\!C} \, (1+\varepsilon) \, \bar{y}\},0\right\}.
\end{align*}
The value of \(\varepsilon_{y\bar{y}}^C\) corresponds to the smallest value of \(\varepsilon\) such that \(y\) \(\varepsilon\)-dominates \(\bar{y}\) with respect to ordering cone $C$.

This representation relation can be included in the \ref{eq:representationProblem} to specify the \(\varepsilon\)-\ref{eq:representationProblem}. The \(\varepsilon\)-coverage value of a subset \(R\) representing the non-dominated set \(Y_N\) is defined as the smallest \(\varepsilon\) such that every point in \(Y_N\) is \(\varepsilon\)-dominated by at least one point in \(R\). It is computed as
\begin{align*}
    \max_{y\in Y_N} \min_{r\in R} \varepsilon_{ry}^C.
\end{align*}

\begin{definition}[cf.~\citealt{ehrgott-2005}]
    The non-dominated set \(Y_N\) with respect to a ordering cone \(C\) is called \emph{externally stable} if for each \(y\in Y\setminus Y_N\) exists a \(\bar{y}\in Y_N\) such that \(y\in\bar{y}+C\). 
\end{definition}

\begin{lemma} \label{lem:eps-coverage}
    Let \(C\subset\R^p\) be a ordering cone, \(Y\subset C\) be the set of all feasible points of a \ref{eq:MOP_C}, \(Y_N\subseteq Y\) be the non-dominated set w.r.t. the ordering cone \(C\) and \(R\subseteq Y_N\). Let \(Y_N\) be externally stable. Then
    \begin{align*}
        \max_{y\in Y_N} \min_{r\in R} \varepsilon_{ry}^C=\max_{y\in Y} \min_{r\in R} \varepsilon_{ry}^C
    \end{align*}
\end{lemma} 

\begin{proof}
    We show that for every point \(y\in Y\setminus Y_N\), there exists a \(r\in R\) such that \(\varepsilon_{ry}^C\leq\max_{y\in Y_N} \min_{r\in R} \varepsilon_{ry}^C\). Let \(y\in Y\setminus Y_N\). If there is a \(r\in R\) with \(r\preceqq_{\!\!C} \, y\), then \(\varepsilon_{ry}^C=0\). Otherwise, since \(y\notin Y_N\), there exists a \(\bar{y}\in Y_N\) such that \(\bar{y}\preceqq_{\!\!C} \, y\). For any \(\bar{r}\in R\) it holds that \(\varepsilon_{\bar{r}\bar{y}}^C\geq \varepsilon_{\bar{r}y}^C\). Hence, there exists a \(r\in R\) such that \(\varepsilon_{ry}^C\leq\max_{y\in Y_N} \min_{r\in R} \varepsilon_{ry}^C\).
\end{proof}

The \(\varepsilon\)-spacing value of a representation \(R\subseteq Y_N\) is defined as the smallest value of \(\varepsilon\) such that one point in \(R\) \(\varepsilon\)-dominates (w.\,r.\,t.\ the ordering cone \(C\)) another point in \(R\). It is given by 
\begin{align*}
    \min_{\substack{r, \bar{r} \in R \\ r \neq \bar{r}}} \varepsilon_{r\bar{r}}^C.
\end{align*} 

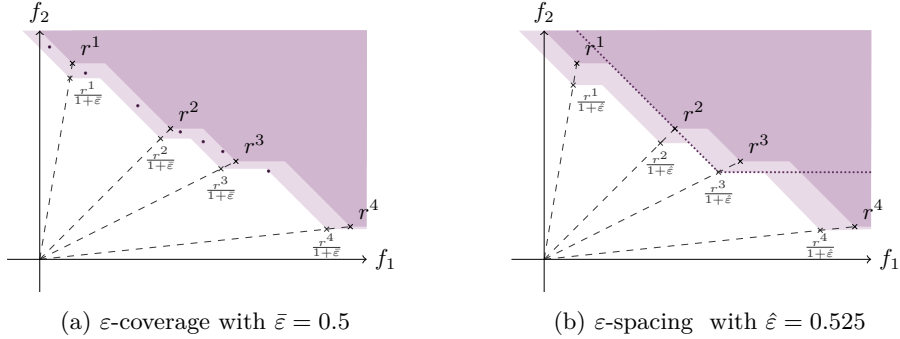
\begin{figure}[htbp]
    \centering
    \begin{subfigure}[t]{0.45\linewidth}
        \centering
        \resizebox{\linewidth}{!}{
            \begin{tikzpicture} [scale = 0.5]
        
            \fill[mycolor!20] (1*9.7/10.5,6*9.7/10.5) -- (-5.6/10.5,7) -- (10,7) -- (10,6*9.7/10.5) -- (1*9.7/10.5,6*9.7/10.5);
            \fill[mycolor!20] (4*9.7/10.5,4*9.7/10.5) -- (4.1/10.5,7) -- (10,7) -- (10,4*9.7/10.5) -- (4*9.7/10.5,4*9.7/10.5);
            \fill[mycolor!20] (6*9.7/10.5,3*9.7/10.5) -- (13.8/10.5,7) -- (10,7) -- (10,3*9.7/10.5) -- (6*9.7/10.5,3*9.7/10.5);
            \fill[mycolor!20] (9.5*9.7/10.5,1*9.7/10.5) -- (28.35/10.5,7) -- (10,7) -- (10,1*9.7/10.5) -- (9.5*9.7/10.5,1*9.7/10.5);
            
            \fill[mycolor!40] (1,6) -- (0,7) -- (10,7) -- (10,6) --(1,6);
            \fill[mycolor!40] (4,4) -- (1,7) -- (10,7) -- (10,4) --(4,4);
            \fill[mycolor!40] (6,3) -- (2,7) -- (10,7) -- (10,3) --(6,3);
            \fill[mycolor!40] (9.5,1) -- (3.5,7) -- (10,7) -- (10,1) -- (9.5,1);
        
            \draw[dashed, black!80] (0,0) -- (1,6);
            \draw[dashed, black!80] (0,0) -- (4,4);
            \draw[dashed, black!80] (0,0) -- (6,3);
            \draw[dashed, black!80] (0,0) -- (9.5,1);
            
            \draw[->] (-1,0) -- (10,0) node[right] {$f_1$};
            \draw[->] (0,-1) -- (0,7) node[above] {$f_2$};
        
            \filldraw[] (1,6)  node[above right] {$r^1$};
            \filldraw[] (4,4) node[above right] {$r^2$};
            \filldraw[] (6,3) node[above right] {$r^3$};
            \filldraw[] (9.5,1) node[above right] {$r^4$};
            \draw plot[only marks, mark=x, mark size=3pt] coordinates {(1,6) (4,4) (6,3) (9.5,1)};
            \draw[draw=mycolor!70!black] plot[only marks, mark=*, mark size=1pt] coordinates {(0.3,6.5) (1.4,5.7) (3,4.7) (4.3,3.9) (5,3.6) (5.6,3.3) (7,2.7)};
            
            \filldraw[black!80] (1*9.7/10.5,6*9.7/10.5)  node[below right, xshift=-1mm, yshift=0.7mm] {\tiny{$\frac{r^1}{1+\bar{\varepsilon}}$}};
            \filldraw[black!80] (4*9.7/10.5,4*9.7/10.5)  node[below] {\tiny{$\frac{r^2}{1+\bar{\varepsilon}}$}};
            \filldraw[black!80] (6*9.7/10.5,3*9.7/10.5) node[below] {\tiny{$\frac{r^3}{1+\bar{\varepsilon}}$}};
            \filldraw[black!80] (9.5*9.7/10.5,1*9.7/10.5) node[below, yshift=0.5mm] {\tiny{$\frac{r^4}{1+\bar{\varepsilon}}$}};

            \draw[black!80] plot[only marks, mark=x, mark size=3pt] coordinates {(1*9.7/10.5,6*9.7/10.5) (4*9.7/10.5,4*9.7/10.5) (6*9.7/10.5,3*9.7/10.5) (9.5*9.7/10.5,1*9.7/10.5)};
        
        \end{tikzpicture}
        }
        \caption{\(\varepsilon\text{-coverage}\) with \(\bar{\varepsilon}=0.5\)}
    \end{subfigure}
    \hfill
    \begin{subfigure}[t]{0.45\linewidth}
        \centering
        \resizebox{\linewidth}{!}{
            \begin{tikzpicture} [scale = 0.5]
\fill[mycolor!20] (1/1.125,6/1.125) -- (-7/9,7) -- (10,7) -- (10,6/1.125) -- (1/1.125,6/1.125);
            \fill[mycolor!20] (4/1.125,4/1.125) -- (1/9,7) -- (10,7) -- (10,4/1.125) -- (4/1.125,4/1.125);
            \fill[mycolor!20] (6/1.125,3/1.125) -- (1,7) -- (10,7) -- (10,3/1.125) -- (6/1.125,3/1.125);
            \fill[mycolor!20] (9.5/1.125,1/1.125) -- (2.5,7) -- (10,7) -- (10,1*9.7/10.5) -- (10/1.125,1/1.125);
        
            \fill[mycolor!40] (1,6) -- (0,7) -- (10,7) -- (10,6) --(1,6);
            \fill[mycolor!40] (4,4) -- (1,7) -- (10,7) -- (10,4) --(4,4);
            \fill[mycolor!40] (6,3) -- (2,7) -- (10,7) -- (10,3) --(6,3);
            \fill[mycolor!40] (9.5,1) -- (3.5,7) -- (10,7) -- (10,1) -- (9.5,1);
        
            \draw[dashed, black!80] (0,0) -- (1,6);
            \draw[dashed, black!80] (0,0) -- (4,4);
            \draw[dashed, black!80] (0,0) -- (6,3);
            \draw[dashed, black!80] (0,0) -- (9.5,1);

            \draw[densely dotted, thick, mycolor!70!black] (10,3/1.125) -- (6/1.125,3/1.125);
            \draw[densely dotted, thick, mycolor!70!black] (1,7) -- (6/1.125,3/1.125);
            
            \draw[->] (-1,0) -- (10,0) node[right] {$f_1$};
            \draw[->] (0,-1) -- (0,7) node[above] {$f_2$};
            \filldraw[] (1,6)  node[above right] {$r^1$};
            \filldraw[] (4,4) node[above right] {$r^2$};
            \filldraw[] (6,3) node[above right] {$r^3$};
            \filldraw[] (9.5,1) node[above right] {$r^4$};
            \draw plot[only marks, mark=x, mark size=3pt] coordinates {(1,6) (4,4) (6,3) (9.5,1)};
            
            \filldraw[black!80] (1/1.125,6/1.125) node[below right, xshift=-1mm, yshift=0.7mm] {\tiny{$\frac{r^1}{1+\hat{\varepsilon}}$}};
        
            \filldraw[black!80] (4/1.125,4/1.125) node[below] {\tiny{$\frac{r^2}{1+\hat{\varepsilon}}$}};
            
            \filldraw[black!80] (6/1.125,3/1.125) node[below] {\tiny{$\frac{r^3}{1+\hat{\varepsilon}}$}};
        
            \filldraw[black!80] (9.5/1.125,1/1.125) node[below, yshift=0.7mm] {\tiny{$\frac{r^4}{1+\hat{\varepsilon}}$}};

            \draw[black!80] plot[only marks, mark=x, mark size=3pt] coordinates {(1/1.125,6/1.125) (4/1.125,4/1.125) (6/1.125,3/1.125) (9.5/1.125,1/1.125)};
        
        \end{tikzpicture}
        }
        \caption{\(\varepsilon\text{-spacing }\) with \(\hat{\varepsilon}=0.525\)}
    \end{subfigure}

    \caption{Representation \(R=\{r^1,r^2,r^3,r^4\}\) (marked with \(\times\)) and the other non-dominated points \(Y_N\setminus R\) (marked with \(\boldsymbol{\cdot}\)). The violet area depicts \(\bigcup_i (r^i+C)\) while the light violet area corresponds to \(\bigcup_i (\frac{r^i}{1+\varepsilon}+C)\).}
    \label{fig:epsRep}
\end{figure} 
In Figure~\ref{fig:epsRep} an example for the \(\varepsilon\)-coverage value and the \(\varepsilon\)-spacing value for a given non-dominated set and a given representation is illustrated.

\begin{lemma}\label{lem:eps-equ}
    Let \(C\subset\R^p\) be a polyhedral ordering cone, \(y,\bar{y}\in\R^p\), \(\varepsilon\geq 0\) and \(A_C\in\R^{m\times p}\), with \(C=\hcone(A_C)\). It holds:
    \begin{align*}
    y\preceqq_{\!\!C,\varepsilon} \, \bar{y} 
    \iff A_C \cdot y \epreceqq \,A_C\cdot \bar{y}
    \end{align*}
\end{lemma}

\begin{proof}\quad
    \(\begin{aligned}[t]
    y\preceqq_{\!\!C,\varepsilon} \,\bar{y} &\stackrel{\text{Def.~\ref{def:eps-dom}}}{\iff} 
    y \preceqq_{\!\!C} \,(1+\varepsilon) \, \bar{y} \\
    & \iff A_C \cdot y \preceqq (1+\varepsilon) \,A_C\cdot \bar{y} \\
    &\stackrel{\text{Def.~\ref{def:eps-dom-pareto}}}{\iff} A_C \cdot y \epreceqq \,A_C\cdot \bar{y}
    \end{aligned}\)\\[-2.5ex]\qedhere
\end{proof}

Lemma~\ref{lem:eps-equ} shows that the \(\varepsilon\)-representation relation for a optimization problem with a polyhedral ordering cone and the corresponding Pareto optimization problem are equivalent, i.\,e., \(\varepsilon^C_{y\bar{y}}=\varepsilon_{A_C\cdot y,A_C\cdot \bar{y}}\).

We conclude that solving the \(\varepsilon\)-\ref{eq:representationProblem} for \ref{eq:MOP_C} is equivalent to the solution of the \(\varepsilon\)-\ref{eq:representationProblem} on the corresponding MOP, which is generated by a linear transformation of the \ref{eq:MOP_C} with the matrix \(A_C\):
\begin{align*}
    \max_{y\in Y_N}\; \min_{r\in R} \varepsilon_{ry}^C&= \max_{y\in (A_C\cdot Y)_N} \min_{r\in A_C\cdot R} \varepsilon_{ry}, \\
    \min_{\substack{r, \bar{r} \in R \\ r \neq \bar{r}}} \varepsilon_{r\bar{r}}^C &= \min_{\substack{r, \bar{r} \in A_C\cdot R \\ r \neq \bar{r}}} \varepsilon_{r\bar{r}}.
\end{align*}
Accordingly, the existing theory for the Pareto case can be applied when considering multi-objective optimization problems with polyhedral ordering cones.

\subsection{Representation Relation Based on \(\gamma\)-Dominance} \label{sec:Gamma}

The \(\gamma\)-representation relation is also based on a modification of the ordering cone. In the case of \(\varepsilon\)-dominance, the shape of the cone remains unchanged but is shifted closer to the origin. In contrast, for \(\gamma\)-dominance, the cone is enlarged by expanding its facets `outwards'. 

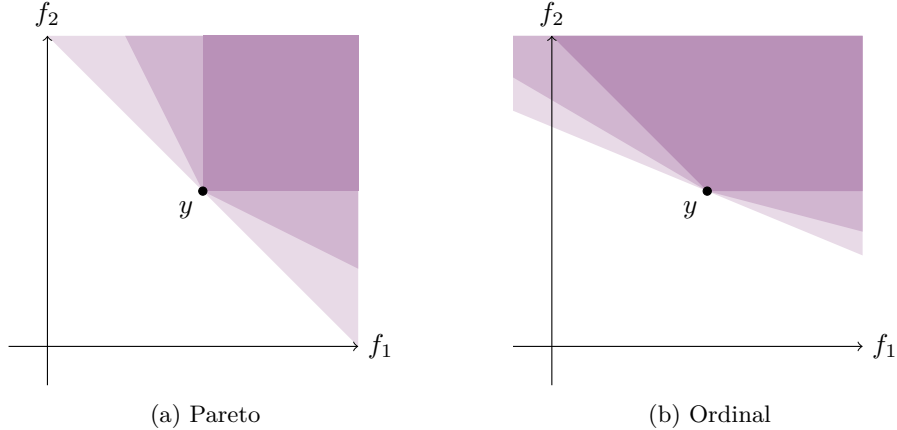
\begin{figure}[htbp]
    \centering

    \begin{subfigure}[t]{0.45\linewidth}
        \centering
        \resizebox{\linewidth}{!}{
            \begin{tikzpicture}
                \fill[mycolor!20] (2,2) -- (4,0) -- (4,4) -- (0,4) -- (2,2);
                \fill[mycolor!40] (2,2) -- (4,1) -- (4,4) -- (1,4) -- (2,2);
                \fill[mycolor!60] (2,2) -- (4,2) -- (4,4) -- (2,4) -- (2,2);
                
                \draw[->] (-0.5,0) -- (4,0) node[right] {$f_1$};
                \draw[->] (0,-0.5) -- (0,4) node[above] {$f_2$};
    
                \filldraw[] (2,2) circle (1.5pt) node[below left] {$y$};
            \end{tikzpicture}
        }
        \caption{Pareto}
        \label{fig:gamma-dom-pareto}
    \end{subfigure}
    \hfill
    \begin{subfigure}[t]{0.45\linewidth}
        \centering
        \resizebox{\linewidth}{!}{
            \begin{tikzpicture}
                \fill[mycolor!20] (2,2) -- (4,{2*sqrt(2)/(sqrt(2)+1)}) -- (4,4) -- (-0.5,4) -- (-0.5, {5/sqrt(2)-0.5}) -- (2,2);
                \fill[mycolor!40] (2,2) -- (4,{sqrt(2)*2/(sqrt(2)+0.5)}) -- (4,4) -- (-0.5,4) -- (-0.5,{(sqrt(2)+4.5)/(sqrt(2)*0.5+1)}) -- (2,2);
                \fill[mycolor!60] (2,2) -- (4,2) -- (4,4) -- (0,4) -- (2,2);
                
                \draw[->] (-0.5,0) -- (4,0) node[right] {$f_1$};
                \draw[->] (0,-0.5) -- (0,4) node[above] {$f_2$};
    
                \filldraw[] (2,2) circle (1.5pt) node[below left] {$y$};
            \end{tikzpicture}
        }
        \caption{Ordinal}
        \label{fig:gamma-dom-ordinal}
    \end{subfigure}

    \caption{$\gamma$-enlarged ordering cones (with $\gamma=0.5$ and $\gamma=1$)\label{fig:gamma-dom}}
\end{figure} 
\begin{definition}[cf.~\citealt{humbertropers25representation}]
    Let \(\gamma\in[0,1]\) and \(y,\bar{y}\in \R^p\). The point \(y\) is said to \(\gamma\)-dominate \(\bar{y}\) with respect to the Pareto cone, denoted by \(y\preceqq_{\!\!\gamma} \,\bar{y}\) if
    \begin{align*}
        y_i \leq \bar{y}_i + \gamma\sum_{\substack{k=1 \\k\neq i }}^p (\bar{y}_k-y_k) \qquad \forall i\in\{1,\dots,p\}.
    \end{align*}
\end{definition}

Note that for \(\gamma=0\) the \(\gamma\)-dominance is equivalent to the Pareto-dominance, except that a point \(\gamma\)-dominates itself. As \(\gamma\) increases, the cone opens symmetrically until \(\gamma=1\), where it corresponds to the weighted sum 
with equal weights. In the following, we will denote the \(\gamma\)-enlarged Pareto cone by \(P_\gamma\). Figure~\ref{fig:gamma-dom-pareto} illustrates the cone in the Pareto case. 

\begin{lemma} \label{lem:gam-dom-Pareto}
    Let \(\gamma\in[0,1]\) and \(y,\bar{y}\in \R^p\). If \(y\preceqq \bar{y}\), then \(y\preceqq_{\!\!\gamma} \,\bar{y}\).
\end{lemma}
\begin{proof}\quad\(\begin{aligned}[t]
    y\preceqq \bar{y} \Longrightarrow& \,y_i \leq \bar{y}_i \quad \forall i\in\{1,\dots,p\}\\
    \Longrightarrow& \, y_i \leq \bar{y}_i + \underbracket[0.5pt]{\gamma}_{\ge 0} \sum_{\substack{k=1 \\k\neq i }}^p \underbracket[0.5pt]{(\bar{y}_k-y_k)}_{\ge 0} \qquad \forall i\in\{1,\dots,p\}
\end{aligned}\)\\[-3ex]\qedhere
\end{proof}

Lemma~\ref{lem:gam-dom-Pareto} shows that the \(\gamma\)-ordering cone contains the Pareto cone and that the \(\gamma\)-dominance is a relaxation of Pareto-dominance for all \(\gamma\in[0,1]\). The concept of \(\gamma\)-dominance can be equivalently written as:
\begin{align*}
    y \preceqq_{\!\!\gamma} \,\bar{y} &\iff y_i \leq \bar{y}_i + \gamma\,\sum_{\substack{k=1 \\k\neq i }}^p (\bar{y}_k - y_k)   \quad \forall i\in\{1,\ldots,p\}\\
    \iff& y_i + \gamma\,\sum_{\substack{k=1 \\k\neq i }}^p  y_k \leq \bar{y}_i + \gamma\,\sum_{\substack{k=1 \\k\neq i }}^p  \bar{y}_k   \quad \forall i\in\{1,\ldots,p\}\\
    \iff&(\gamma, \ldots, \gamma, 
    \underset{\substack{\uparrow\\\mathclap{i\text{-th comp.}}}}{1}, 
    \gamma, \ldots, \gamma) \cdot y 
    \leq  
    (\gamma, \ldots, \gamma, 
    \underset{\substack{\uparrow\\\mathclap{i\text{-th comp.}}}}{1},
    \gamma, \ldots, \gamma) \cdot \bar{y} \qquad \forall i\in\{1,\ldots,p\}
\end{align*}
Using this reformulation, we see that the \(\gamma\)-enlarged Pareto cone \(P_\gamma\) is induced by the matrix \(A_\gamma\), i.\,e., \(P_\gamma=\hcone(A_\gamma)\) with:
\begin{equation*}
    A_\gamma\coloneqq\begin{tikzpicture}[baseline={($(current bounding box.center)- (0,.8ex)$)}]
        \matrix (m) [matrix of math nodes,nodes in empty cells,right delimiter={)},left delimiter={(},inner sep=-2pt,nodes={inner sep=1.2ex}]{
        1 & \gamma &  & \gamma\\
        \gamma&  & &  \\
        & & & \gamma\\
        \gamma&  & \gamma & 1 \\
        };
        \draw[loosely dotted,thick] (m-1-2)-- (m-1-4);
        \draw[loosely dotted,thick] (m-1-1)-- (m-4-4);
        \draw[loosely dotted,thick] (m-1-2)-- (m-3-4);
        \draw[loosely dotted,thick] (m-1-4)-- (m-3-4);
        \draw[loosely dotted,thick] (m-2-1)-- (m-4-3);
        \draw[loosely dotted,thick] (m-2-1)-- (m-4-1);
        \draw[loosely dotted,thick] (m-4-1)-- (m-4-3);
    \end{tikzpicture}\in\R^{p\times p}.
\end{equation*}

Note that there is a fundamental difference between \(\gamma\)- and \(\varepsilon\)-dominance. With respect to the \(\varepsilon\)-dominance, for any two points \(y,\bar{y}\in P\), there exists an \(\varepsilon\geq0\) such that \(y\) \(\varepsilon\)-dominates \(\bar{y}\). This property does not hold for the \(\gamma\)-dominance. In particular, if \(\sum_{k=1}^p y_k> \sum_{k=1}^p \bar{y}_k\) the vector \(y\) does not \(\gamma\)-dominate \(\bar{y}\) for any \(\gamma\in[0,1]\). However, for any two points \(y,\bar{y}\in\R^p\), there exists a value of \(\gamma\in[0,1]\) such that either \(y\preceqq_{\!\!\gamma} \,\bar{y}\) or \(\bar{y}\preceqq_{\!\!\gamma} \, y\). Another difference is that, in contrast to \(\varepsilon\)-dominance, the \(\gamma\)-dominance is defined for all points in \(\R^p\) and not only for those located in the ordering cone.

In the Pareto case, the \emph{degree of acceptability} for \(y\) representing \(\bar{y}\) \emph{based on \(\gamma\)-dominance} is defined by:
\begin{align*}
    \gamma_{y,\bar{y}}=\begin{cases}
        \text{undefined} \quad& \text{if } \sum\limits_{k=1}^p y_k> \sum\limits_{k=1}^p \bar{y}_k \\
        0 \quad &\text{if }  y \preceqq \bar{y} \\
        \displaystyle \max\limits_{i\colon y_i<\bar{y}_i} \biggl\{\frac{y_i-\bar{y}_i}{\sum_{k=1,k\neq i}^{p} (\bar{y}_k-y_k)}\biggr\} \quad &\text{otherwise}
    \end{cases}
\end{align*}
It specifies the smallest value of \(\gamma\) such that the vector \(y\) \(\gamma\)-dominates the vector \(\bar{y}\) if such a value exists. Otherwise, \(\gamma_{y,\bar{y}}\) is undefined. In the following, we want to extend the concept of \(\gamma\)-dominance to general ordering cones.

\begin{definition}
    Let \(\gamma\in[0,1]\), \(y,\bar{y}\in \R^p\) and \(C\subset\R^p\) be a polyhedral ordering cone. We say the point \(y\) \(\gamma\)-dominates \(\bar{y}\) with respect to \(C\) (\(y\preceqq_{\!\!C,\gamma} \, \bar{y}\)) if
    \begin{align*}
        A_C\cdot y \preceqq_{\!\!\gamma} \,A_C \cdot \bar{y},
    \end{align*} with \(A_C\in\R^{m\times p}\) being a row-normalized matrix with \(C=\hcone(A_C)\).
\end{definition}

Like in the previous section, the matrix \(A_C\) is row-normalized since the norm of the rows impacts the \(\gamma\)-enlarged cone, while the ordering cone \(C=\hcone(A_C)\) is invariant w.\,r.\,t.\ it. For any polyhedral cone \(C\) the row-normalized matrix \(A_C\) with \(C=\hcone(A_C)\) is uniquely defined (modulo row permutations). This guarantees, that the \(\gamma\)-dominance with respect to a polyhedral ordering cone is well-defined. Figure~\ref{fig:gamma-dom-ordinal} illustrates the \(\gamma\)-enlarged ordinal cone in the bi-objective case. 

\begin{lemma} \label{lem:gam-dom-general}
    Let \(\gamma\in[0,1]\), \(y,\bar{y}\in \R^p\), \(C\subset\R^p\) be a polyhedral ordering cone and \(A_C\in\R^{m\times p}\) a row-normalized matrix, with \(\hcone(A_C)=C\). If \(y\preceqq_{\!\!C} \,\bar{y}\), then \(y\preceqq_{\!\!C,\gamma} \,\bar{y}\).
\end{lemma}
\begin{proof}\(\begin{aligned}[t]
    y\preceqq_{\!\!C} \,\bar{y} \iff& A_C\cdot y \preceqq A_C\cdot \bar{y} \\
    \stackrel{\text{Lem.~\ref{lem:gam-dom-Pareto}}}{\Longrightarrow}& \, A_C\cdot y \preceqq_{\!\!\gamma} \, A_C\cdot \bar{y} \iff y \preceqq_{\!\!C,\gamma} \, \bar{y} 
\end{aligned}\)\\[-2ex]
\end{proof}
Lemma~\ref{lem:gam-dom-general} shows that the cone-dominance is a stronger relation, then the \(\gamma\)-dominance and accordingly, the ordering cone \(C\) is contained in its \(\gamma\)-enlarged cone, which we will denote by \(C_\gamma\). Using the reformulation of the \(\gamma\)-dominance with respect to the Pareto cone, a point \(y\) \(\gamma\)-dominates \(\bar{y}\) with \(\gamma\in[0,1]\) and \(y,\bar{y}\in\R^p\), if and only if

\begin{align} \label{eq:gammaReformulation}
    (\gamma, \ldots, \gamma, 
    \underset{\substack{\uparrow\\\mathclap{i\text{-th comp.}}}}{1}, 
    \gamma, \ldots, \gamma) \cdot A_C \cdot y \leq  (\gamma, \ldots, \gamma, \underset{\substack{\uparrow\\\mathclap{i\text{-th comp.}}}}{1}, \gamma, \ldots, \gamma) \cdot A_C \cdot \bar{y}  \quad \forall i\in\{1,\ldots,m\}.
\end{align} 
This reformulation implies that the \(\gamma\)-enlarged cone \(C_\gamma\) of a polyhedral ordering cone \(C\) is given by \(C_\gamma=\hcone(A_\gamma\cdot A_C)\) with \(A_\gamma\) given by
\begin{equation*}
    A_\gamma\coloneqq\begin{tikzpicture}[baseline=(current bounding box.center)]
        \matrix (m) [matrix of math nodes,nodes in empty cells,right delimiter={)},left delimiter={(},inner sep=-2pt,nodes={inner sep=1.2ex}]{
        1 & \gamma &  & \gamma\\
        \gamma&  & &  \\
        & & & \gamma\\
        \gamma&  & \gamma & 1 \\
        };
        \draw[loosely dotted,thick] (m-1-2)-- (m-1-4);
        \draw[loosely dotted,thick] (m-1-1)-- (m-4-4);
        \draw[loosely dotted,thick] (m-1-2)-- (m-3-4);
        \draw[loosely dotted,thick] (m-1-4)-- (m-3-4);
        \draw[loosely dotted,thick] (m-2-1)-- (m-4-3);
        \draw[loosely dotted,thick] (m-2-1)-- (m-4-1);
        \draw[loosely dotted,thick] (m-4-1)-- (m-4-3);
    \end{tikzpicture}\in\R^{m\times m}.
\end{equation*}
For example the \(\gamma\)-enlarged ordinal cone in the bi-objective case is given by 

\begin{align*}
    \hcone\left(\begin{pmatrix}
        1 & \gamma\\
        \gamma & 1
    \end{pmatrix}\cdot \ \begin{pmatrix}
        \frac{1}{\sqrt{2}} & \frac{1}{\sqrt{2}}\\
        0 & 1
    \end{pmatrix} \right) = \hcone\begin{pmatrix}
        \frac{1}{\sqrt{2}} & \frac{1}{\sqrt{2}}+\gamma\\
        \frac{\gamma}{\sqrt{2}} & \frac{\gamma}{\sqrt{2}}+1
    \end{pmatrix}.
\end{align*}
Additionally, for \(\gamma=1\) the \(\gamma\)-enlarged cone \(C_\gamma\) is always a half-space. For the ordinal cone the normal vector of this halfspace \(C_\gamma\) with \(\gamma=1\) is given by

\begin{equation*}
    \begin{pmatrix}
    1 & \cdots & 1
    \end{pmatrix}
\begin{tikzpicture} [baseline=(current bounding box.center)]
    \matrix (n) [
        matrix of math nodes,
        nodes in empty cells,
        right delimiter={)},
        left delimiter={(},
        inner sep=-2pt,
        nodes={inner sep=1ex}
    ]{
        \frac{1}{\sqrt{p}} & & & \frac{1}{\sqrt{p}} & \frac{1}{\sqrt{p}}\\
         & \frac{1}{\sqrt{p-1}} & & \frac{1}{\sqrt{p-1}}& \frac{1}{\sqrt{p-1}}\\
        & & & & \\
        & & & \frac{1}{\sqrt{2}} & \frac{1}{\sqrt{2}}\\
        0 & & & & 1\\
    };

    \draw[loosely dotted,thick] (n-1-1)--(n-1-4);
    \draw[loosely dotted,thick] (n-2-2)--(n-4-4);
    \draw[loosely dotted,thick] (n-2-2)--(n-2-4);
    \draw[loosely dotted,thick] (n-2-4)--(n-4-4);
    \draw[loosely dotted,thick] (n-2-5)--(n-4-5);
\end{tikzpicture} = 
\biggl( \textstyle\frac{1}{\sqrt{p}}, \;\sum\limits_{i=p-1}^p \frac{1}{\sqrt{i}},\; \ldots, \;\sum\limits_{i=1}^p \frac{1}{\sqrt{i}} \biggr)
\end{equation*}

In the following, we investigate how ordering cones are opened, when considering the \(\gamma\)-dominance. For this purpose we determine the opening angles in the bi-objective case, i.\,e., the angle \(\theta_j\), \(j=1,2\) between normal vectors of corresponding facets of \(A_C\) and \(A_\gamma\, A_C\), respectively.

\begin{lemma}
    Let \(\gamma\in[0,1]\), \(C\subset\R^2\) be a polyhedral ordering cone and let \(A_C\in\R^{m\times p}\) be a row-normalized matrix with \(\hcone(A_C)=C\). For \(j=1,2\) let \(\theta_j\) denote the angle between the \(j\)-th row of matrix \(A_C\) and the \(j\)-th row of matrix \(A_\gamma\, A_C\). Then, \(\theta_1=\theta_2\). 
\end{lemma}
\begin{proof}
    The \(\gamma\)-enlarged ordering cone is defined by the following matrix
    {\footnotesize
    \begin{align*}
        A'\coloneqq A_\gamma\cdot A_C=\begin{pmatrix}
        a_{11}+\gamma \,a_{21} & a_{12}+\gamma \,a_{22} \\
        \gamma\, a_{11}+a_{21} & \gamma\, a_{12}+a_{22}
    \end{pmatrix}, \text{ with } A_\gamma=\begin{pmatrix}
        1 & \gamma\\
        \gamma & 1
    \end{pmatrix}, 
    A_C=\begin{pmatrix}
        a_{11} & a_{12}\\
        a_{21} & a_{22}
    \end{pmatrix}.
    \end{align*}}
    The angle \(\theta_1\) can be calculated by solving
    \begin{align*}
        \cos (\theta_1)=\frac{\langle a_1, a_1'\rangle}{\underbracket[0.5pt]{\lVert a_1\rVert}_{=1} \cdot\,\lVert a_1'\rVert},
    \end{align*}
    with \(a_1\coloneqq(a_{11}, a_{12})\) and \(a'_1\coloneqq (a_{11}+\gamma\,a_{21}, \; a_{12}+\gamma\,a_{22})\).
    We will now examine the individual components:
    \begin{align*}
        \langle a_1, a_1'\rangle&=a_{11}^2+\gamma\, a_{11}\,a_{21} + a_{12}^2+\gamma \,a_{12}\,a_{22}\\
        &\stackrel{\lVert a_1\rVert=1}{=}1+\gamma\, (a_{11}\,a_{21}+a_{12}a_{22}),
    \end{align*}
    \begin{align*}
        \lVert a_1'\rVert^2 &= a_{11}^2+2\,\gamma\, a_{11}\,a_{21}+\gamma^2\,a_{21}^2+a_{12}^2+2\,\gamma\, a_{12}\,a_{22}+\gamma^2\,a_{22}^2\\
&=1+2\,\gamma \,(a_{11}a_{21}+a_{12}\,a_{22}) +\gamma^2.
    \end{align*}
    With the analogous computation for \(\theta_2\) we obtain:
    \begin{align*}
        \cos (\theta_1) = \frac{1+\gamma\, (a_{11}a_{21}+a_{12}a_{22})}{\sqrt{1+2\,\gamma\,( a_{11}\,a_{21}+a_{12}\,a_{22}) +\gamma^2}}= \cos (\theta_2).
    \end{align*} 
Thus, for any value of \(\gamma\) the two opening angles are equal, i.\,e., \(\theta_1=\theta_2\). 
\end{proof}

Accordingly, in the bi-objective case, polyhedral  ordering cones are opened symmetrically as \(\gamma\) increases. When \(C\subset\R^p\), with \(p>2\) this result does not hold. This can be shown, for example, for the ordinal cone. The ordinal cone \(C_o\) in \(\R^3\) is generated by the row-normalized matrix
\begin{align*}
    A_{C_o}=\begin{pmatrix}
        \frac{1}{\sqrt{3}} & \frac{1}{\sqrt{3}} & \frac{1}{\sqrt{3}}\\
        0 & \frac{1}{\sqrt{2}} & \frac{1}{\sqrt{2}}\\
        0 & 0& 1
    \end{pmatrix}. 
\end{align*} 
The reformulation \eqref{eq:gammaReformulation} of the \(\gamma\)-dominance w.\,r.\,t.\ the ordering cone \(C\) implies that the ordering cone is enlarged to a half-space for \(\gamma=1\). However, there isn't a half-space for which the angles between the three row-vectors of \(A_{C_o}\) and the half-space are equal.
This raises the question which ordering cones admit a half-space for which the angles between the half-space and all normal vectors of the cone are equal. Such a half-space exists only for ordering cones possessing a \emph{symmetry axis} \(S\), defined as the sum of the normal vectors \(a_i\) of the cone,
\begin{align*}
    S=\sum_{i=1}^m a_i.
\end{align*}
In other words, the angle between the symmetry axis \(S\) and each normal vector \(a_i\) of \(A_C\in\R^{m\times p}\) is the same. In this case, is the half-space obtained for \(\gamma=1\), and does the ordering cone remain symmetrically opened for all other values of \(\gamma\)?

\begin{lemma}
    Let \(C\subset\R^p, \, p>2\) be a polyhedral ordering cone, \(A_C\in\R^{m\times p}\) a row-normalized matrix with \(C=\hcone(A_C)\) and \(\gamma\in[0,1]\). If the cone has a symmetry axis \(S\), the \(\gamma\)-enlarged cone of \(C\) is a symmetrical opening of \(C\).
\end{lemma}

\begin{proof}
    Let \(a_1,\dots, a_m\) be the rows of \(A_C\), i.\,e., 
    \begin{align*}
        A_C=\begin{pmatrix}
            a_1\\
            \vdots\\
            a_m
        \end{pmatrix}.
    \end{align*} It holds \(\lVert a_i\rVert=1, \forall\ i\in\{1,\dots,m\}\) and 
    \begin{align*}
        \exists \kappa\in\R: \frac{\langle S, a_i\rangle}{\lVert S\rVert\cdot\lVert a_i\rVert}=\kappa, \ \forall i\in\{1,\dots,m\}.
    \end{align*} Consequently,
    \begin{align*}
        \exists \bar{\kappa}\in\R: \langle S, a_i\rangle=\bar{\kappa}, \ \forall i\in\{1,\dots,m\}.
    \end{align*} The \(\gamma\)-enlarged cone is given by:
    \begin{equation*}
        A_\gamma\cdot A_C=
        \begin{tikzpicture}[baseline=(current bounding box.center)]
        \matrix (m) [matrix of math nodes,nodes in empty cells,right delimiter={)},left delimiter={(},inner sep=-2pt,nodes={inner sep=1.2ex}]{
        1 & \gamma &  & \gamma\\
        \gamma&  & &  \\
        & & & \gamma\\
        \gamma&  & \gamma & 1 \\
        };
        \draw[loosely dotted,thick] (m-1-2)-- (m-1-4);
        \draw[loosely dotted,thick] (m-1-1)-- (m-4-4);
        \draw[loosely dotted,thick] (m-1-2)-- (m-3-4);
        \draw[loosely dotted,thick] (m-1-4)-- (m-3-4);
        \draw[loosely dotted,thick] (m-2-1)-- (m-4-3);
        \draw[loosely dotted,thick] (m-2-1)-- (m-4-1);
        \draw[loosely dotted,thick] (m-4-1)-- (m-4-3);
    \end{tikzpicture} \cdot 
    \begin{tikzpicture}[baseline=(current bounding box.center)]
        \matrix (m) [matrix of math nodes,nodes in empty cells,right delimiter={)},left delimiter={(},inner sep=-2pt,nodes={inner sep=1.2ex}]{
        a_1 \\
        \\
        a_m\\
        };
        \draw[loosely dotted,thick] (m-1-1)-- (m-3-1);
    \end{tikzpicture}= \begin{tikzpicture}[baseline=(current bounding box.center)]
        \matrix (m) [matrix of math nodes,nodes in empty cells,right delimiter={)},left delimiter={(},inner sep=-2pt,nodes={inner sep=1.2ex}]{
        a_1+\gamma\,\sum_{i=1, i\neq 1}^m a_i\\
        \\
        a_m+\gamma\,\sum_{i=1, i\neq m}^m a_i\\
        };
        \draw[loosely dotted,thick] (m-1-1)-- (m-3-1);
    \end{tikzpicture}.
    \end{equation*}
    The angle \(\theta_j\) between the original normal vector \(a_j\) and the normal vector of the associated facet of \(C_\gamma\), i.e., \(a_j+\gamma\,\sum_{i=1, i\neq m}^m a_i\) can be computed by solving:
    \begin{align*}
        \cos(\theta_j)
        =\frac{\bigl\langle a_j, \; a_j+\gamma\,\sum_{i=1, i\neq j}^m a_i\bigr\rangle}{\lVert a_j\rVert\cdot \bigl\lVert  a_j+\gamma\, \sum_{i=1, i\neq j}^m a_i \bigr\rVert}.
    \end{align*}
    We want to show, that \(\theta_1=\ldots=\theta_p\) by proving that the value of this term is independent of \(j\). We start with the numerator:
    
    \begin{align*}
        \biggl\langle a_j, \;a_j+\gamma \sum_{\substack{i=1\\ i\neq j}}^m a_i\biggr\rangle 
        = \underbracket[0.5pt]{\lVert a_j\rVert^2}_{=1}
        + \ \gamma\,\Bigl\langle a_j, \underbracket[0.5pt]{\sum_{\substack{i=1}}^m a_i}_{=S}\Bigr\rangle -\gamma\underbracket[0.5pt]{\lVert a_j\rVert^2}_{=1} =1+\gamma \,(\bar{\kappa}-1)
    \end{align*}
    Accordingly, the numerator is independent from \(j\). We will show the same for the denominator:
    \begin{align*}
        \Bigl\lVert &a_j+\gamma \,\sum_{\substack{i=1\\ i\neq j}}^m a_i \Bigr\rVert^2
        =\underbracket[0.5pt]{\lVert a_j\rVert^2}_{=1} +\ 2\,\gamma\, \underbracket[0.5pt]{\Bigl\langle a_j, \sum_{\substack{i=1\\ i\neq j}}^m a_i \Bigr\rangle}_{(\star)} 
        + \underbracket[0.5pt]{\Bigl\lVert\gamma \sum_{\substack{i=1\\ i\neq j}}^m a_i \Bigr\rVert^2}_{(\ast)}
    \end{align*}
    
    \begin{align*}
    (\star):&&&
        \Bigl\langle a_j, \sum_{\substack{i=1\\ i\neq j}}^m a_i \Bigr\rangle
        = \Bigl\langle a_j, \underbracket[0.5pt]{\sum_{i=1}^m a_i}_{=S} \Bigr\rangle -\lVert a_j\rVert^2= \bar{\kappa}-1\\
    (\ast): &&&
        \Bigl\lVert\gamma \sum_{\substack{i=1\\ i\neq j}}^m a_i \Bigr\rVert^2 
        =\gamma\, \Bigl(\Bigl\lVert\sum_{i=1}^m a_{i} \Bigr\rVert^2-
        2\,\Bigl\langle\underbracket[0.5pt]{\sum_{i=1}^m a_i, a_j\Bigr\rangle\!\!}_{=\bar{x}}\; + \underbracket[0.5pt]{\lVert a_j\rVert^2\!\!}_{=1}\;\Bigr)= \gamma \bigl(\lVert S\rVert^2-2\,\bar{\kappa}+1 \bigr)
    \end{align*}
Thus, the denominator is also independent of \(j\). Consequently, \(\theta_1=\dots=\theta_p\) and the ordering cone \(C\) opens symmetrically for increasing values of \(\gamma\).
\end{proof}

For a polyhedral ordering cone \(C\), the \emph{degree of acceptability based on \(\gamma\)-dominance} for \(y\) representing \(\bar{y}\) is defined as
\begin{align*}
    \gamma_{y\bar{y}}^C\coloneqq\gamma_{A_C\cdot y,A_C\cdot \bar{y}},
\end{align*} with \(A_C\in\R^{m\times p}\), defining matrix \(C\), i.\,e., \(C=\hcone(A_C)\).
Similarly to the degree of acceptability based on \(\varepsilon\)-dominance, \(\gamma_{y\bar{y}}^C\) is the smallest value of \(\gamma\) such that \(y\) is \(\gamma\)-dominating \(\bar{y}\). The \(\gamma\)-coverage value for a given non-dominated set \(Y_N\) and a subset \(R\subseteq Y_N\) representing it, is defined as the smallest value of \(\gamma\) such that every point in \(Y_N\) is \(\gamma\)-dominated by at least one point in \(R\). It is computed as
\begin{align*}
    \max_{y\in Y_N}\min_{r\in R} \gamma_{ry}^C.
\end{align*}
Analogously to Lemma~\ref{lem:eps-coverage}, for externally stable sets \(Y\) it holds that
\begin{align*}
    \max_{y\in Y_N}\min_{r\in R} \gamma_{ry}^C= \max_{y\in Y}\min_{r\in R} \gamma_{ry}^C.
\end{align*}

The \(\gamma\)-spacing value of a given representation \(R\subseteq Y_N\) is defined as the smallest value of \(\gamma\) such that at least one point in \(R\) \(\gamma\)-dominates another point in \(R\), i.\,e.,
\begin{align*}
    \min_{r,\bar{r}\in R} \gamma_{r\bar{r}}^C.
\end{align*} 
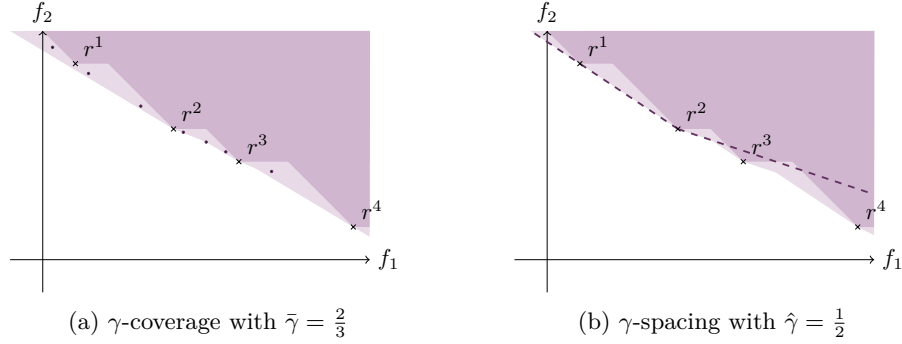
\begin{figure}[htbp]
    \centering
    \begin{subfigure}[t]{0.45\linewidth}
        \centering
        \resizebox{\linewidth}{!}{
            \begin{tikzpicture} [scale = 0.5]
            \fill[mycolor!20] (1,6) -- (-2/3,7) -- (10,7) -- (10,2.4) -- (1,6);
            \fill[mycolor!20] (4,4) -- (-1,7) -- (10,7) -- (10,1.6) -- (4,4);
            \fill[mycolor!20] (6,3) -- (-2/3,7) -- (10,7) -- (10,1.4) -- (6,3);
            \fill[mycolor!20] (9.5,1) -- (-0.5,7) -- (10,7) -- (10,0.7) -- (9.5,1);
            
            \fill[mycolor!40] (1,6) -- (0,7) -- (10,7) -- (10,6) --(1,6);
            \fill[mycolor!40] (4,4) -- (1,7) -- (10,7) -- (10,4) --(4,4);
            \fill[mycolor!40] (6,3) -- (2,7) -- (10,7) -- (10,3) --(6,3);
            \fill[mycolor!40] (9.5,1) -- (3.5,7) -- (10,7) -- (10,1) -- (9.5,1);
            
            \draw[->] (-1,0) -- (10,0) node[right] {$f_1$};
            \draw[->] (0,-1) -- (0,7) node[above] {$f_2$};
        
            \filldraw[] (1,6) node[above right] {$r^1$};
            \filldraw[] (4,4) node[above right] {$r^2$};
            \filldraw[] (6,3) node[above right] {$r^3$};
            \filldraw[] (9.5,1) node[above right] {$r^4$};
            
            \draw[draw=mycolor!70!black] plot[only marks, mark=*, mark size=1pt] coordinates {(0.3,6.5) (1.4,5.7) (3,4.7) (4.3,3.9) (5,3.6) (5.6,3.3) (7,2.7)};
            \draw plot[only marks, mark=x, mark size=3pt] coordinates {(1,6) (4,4) (6,3) (9.5,1)};
        \end{tikzpicture}
        }
        \caption{$\gamma\text{-coverage}$ with \(\bar{\gamma}=\frac{2}{3}\)}
    \end{subfigure}
    \hfill
    \begin{subfigure}[t]{0.45\linewidth}
        \centering
        \resizebox{\linewidth}{!}{
            \begin{tikzpicture} [scale = 0.5]
            \fill[mycolor!20] (1,6) -- (-0.5,7) -- (10,7) -- (10,3) -- (1,6);
            \fill[mycolor!20] (4,4) -- (-0.5,7) -- (10,7) -- (10,2) -- (4,4);
            \fill[mycolor!20] (6,3) -- (0,7) -- (10,7) -- (10,5/3) -- (6,3);
            \fill[mycolor!20] (9.5,1) -- (0.5,7) -- (10,7) -- (10,0.75) -- (9.5,1);
        
            \fill[mycolor!40] (1,6) -- (0,7) -- (10,7) -- (10,6) --(1,6);
            \fill[mycolor!40] (4,4) -- (1,7) -- (10,7) -- (10,4) --(4,4);
            \fill[mycolor!40] (6,3) -- (2,7) -- (10,7) -- (10,3) --(6,3);
            \fill[mycolor!40] (9.5,1) -- (3.5,7) -- (10,7) -- (10,1) -- (9.5,1);
            
            \draw[->] (-1,0) -- (10,0) node[right] {$f_1$};
            \draw[->] (0,-1) -- (0,7) node[above] {$f_2$};

            \draw[dashed, thick, mycolor!70!black] (4,4) -- (-0.5,7);
            \draw[dashed, thick, mycolor!70!black] (4,4) -- (10,2);
            
            \filldraw[] (1,6)  node[above right] {$r^1$};
            \filldraw[] (4,4) node[above right] {$r^2$};
            \filldraw[] (6,3) node[above right] {$r^3$};
            \filldraw[] (9.5,1) node[above right] {$r^4$};
            \draw plot[only marks, mark=x, mark size=3pt] coordinates {(1,6) (4,4) (6,3) (9.5,1)};
                    
        \end{tikzpicture}
        }
        \caption{\(\gamma\text{-spacing}\) with \(\hat{\gamma}=\frac{1}{2}\)}
    \end{subfigure}

    \caption{Representation \(R=\{r^1,r^2,r^3,r^4\} \) (marked with \(\times\)) and the other non-dominated points \(Y_N\setminus R\) (marked with \(\boldsymbol{\cdot}\)). The violet area depicts \(\bigcup_i (r^i+C)\) while the light violet area corresponds to \(\bigcup_i (r^i+C_\gamma)\). Note that \(R\) is not an optimal representation, as \(R'=\{r^2, r^3, r^4\}\) attains the same \(\gamma\)-coverage value and a better \(\gamma\)-spacing value.
    \label{fig:gammaRep}}
\end{figure} 
In Figure~\ref{fig:gammaRep} the \(\gamma\)-coverage value and \(\gamma\)-spacing value is illustrated for a bi-objective \ref{eq:MOP_C}, with the ordinal cone as ordering cone. Note that there does not necessarily exist a \(\gamma\in[0,1]\) for every subset \(R\subseteq Y_N\) such that each point in \(Y_N\) is \(\gamma\)-dominated by at least one point \(r\in R\). However, for any non-dominated set \(Y_N\) and any ordering cone \(C\), there exists a point \(y\in Y_N\) which \(\gamma\)-dominates all the other points in \(Y_N\) with respect to \(C_\gamma\) for some value of \(\gamma\in[0,1]\). This is due to the fact that for \(\gamma=1\) the optimization w.\,r.\,t.\ \(C_\gamma\) corresponds to the optimization of a weighted sum of the objectives (with weights given by the ordering cone \(C\)), whose optimal solution is attained in one of the efficient solutions (mapping to a point \(y\in Y_N\)).
Consequently, there always exists a representation for which the \(\gamma\)-coverage value is well-defined.

From the definition of \(\gamma\)-based degree of acceptability it follows
\begin{align*}
    \max_{y\in Y_N} \min_{r\in R} \gamma_{ry}^C&= \max_{y\in (A_C\cdot Y)_N} \min_{r\in A_C\cdot R} \gamma_{ry}, \\
    \min_{\substack{r, \bar{r} \in R \\ r \neq \bar{r}}} \gamma_{r\bar{r}}^C &= \min_{\substack{r, \bar{r} \in A_C\cdot R \\ r \neq \bar{r}}} \gamma_{r\bar{r}}.
\end{align*}
Consequently, the \(\gamma\)-\ref{eq:representationProblem} for \ref{eq:MOP_C} can be reformulated into a \(\gamma\)-\ref{eq:representationProblem} for a Pareto optimization problem.

\subsection{Point-based Representation Relations for Polyhedral Ordering Cones}

Throughout this chapter, several representation relations have been discussed to evaluate their suitability for finding representations for non-dominated sets of multi-objective optimization problems with polyhedral ordering cones. Thereby, the study focused particularly on the equivalence of representation relations for polyhedral cones with the representation relation for the associated Pareto optimization problem. This would allow us to apply existing representation methods developed for the Pareto case to this more general setting by transforming the cone optimization problem into a Pareto optimization problem, and, subsequently, transforming the resulting representation back into the general cone setting.

For the distance-based relations, a reformulation of \ref{eq:representationProblem} in the Pareto structure is generally not equivalent to the original representation problem. However, the newly introduced cone-induced distance comply with the reformulation. Both the \(\varepsilon\)-dominance and \(\gamma\)-dominance representation relations yield equivalent results for the cone optimization problem and for the associated Pareto optimization problem. When using either of the discussed representation relations, considering polyhedral ordering cones does not impose any additional restrictions compared to the Pareto cone. In the case of \(\varepsilon\)-dominance, the non-dominated set has to be contained in the ordering cone, i.\,e., \(Y_N\subset C\). However, this restriction also applies in the Pareto case. 

A comprehensive analysis of multiplicative representation relations (namely, \(\varepsilon\)-do\-mi\-nance and \(\gamma\)-do\-mi\-nance) is provided in \cite{humbertropers25representation}, where the \(\varepsilon\)-do\-mi\-nance relation is identified as the most suitable relation for representation problems.

\section{Outlook} \label{sec:Outlook}

The aim of this work was to investigate whether representation problems for multi-objective optimization problems with a polyhedral ordering cone can be reformulated as representation problems of the corresponding Pareto optimization problem. We adopted and generalized the formulation introduced in the article by \cite{humbertropers25representation}, which extends the three-dimensional framework proposed by \cite{sayn2000representation}.

Besides representation problems based on a point-based representation relation, representations can also be defined in terms of the hypervolume indicator. In this setting, the quality of a representation is assessed by how well the representation set covers the non-dominated set as a whole, rather than by considering the representation of individual non-dominated points. \cite{emmerich13cone} describes this approach for multi-objective optimization problems with general ordering cones. In the setting considered here, the corresponding hypervolume is partitioned into parallelepipeds instead of the cuboids used in the Pareto case. In \cite{klamroth2023reformulation} it is shown that this hypervolume representation problem for the cone optimization problem is equivalent to the hypervolume problem for the Pareto optimization problem. 

Future research could investigate representation problems for multi-objective optimization problems with non-polyhedral ordering cones~\citep{engau-2009}. An interesting question is whether the representation problem can still be transformed or approximated w.\,r.\,t.\ a corresponding Pareto optimization problem or whether alternative methods need to be developed. 

It is also interesting to consider multi-objective optimization problems with variable ordering structures~\citep{eichfelder-2014,deb23generalized,durea25vector}. In this setting, the shape of the set dominated by a vector \(y\in Y\) depends on \(y\). In general, the dominance structure is not transitive anymore, which leads to new challenges in determining a representation.

Another approach is the representation in the decision space. This could provide the decision maker not only with solutions having different objective values but also with different decisions leading to the same or similar non-dominated points. From this perspective, both the objective space and the decision space should be taken into account when selecting a representation.

Overall, these directions highlight the potential, on the one hand, to expand the definition of representations and, on the other hand, to find representations for additional multi-criteria optimization problems with more complex ordering structures.

\subsection*{Acknowledgements}\noindent
This work is partially supported by the  project SAFeR: Safe, Algorithm-based Footpath Routing funded by the EU program EFRE/JTF NeueWege.IN.NRW (EFRE 20800941).

\bibliographystyle{abbrvnat}

\end{document}